\documentclass[oneside,11pt]{amsart}
\usepackage{amsmath}
\usepackage{amssymb}
\usepackage{graphicx}
\usepackage{tikz}
\usepackage[colorlinks=true]{hyperref}
\usepackage{enumitem}
 \usepackage{float} 
 \usepackage{caption}
\usepackage{color}
\usepackage{soul}

\newtheorem{thm}{Theorem}[section]
\newtheorem{prop}[thm]{Proposition}
\newtheorem{lem}[thm]{Lemma}
\newtheorem{cor}[thm]{Corollary}

\theoremstyle{definition}
\newtheorem{defn}[thm]{Definition}
\newtheorem{rem}[thm]{Remark}

\newtheorem{prob}[thm]{Problem}

\newcommand{\abs}[1]{\lvert{#1}\rvert}
\newcommand{\bigabs}[1]{\bigl|{#1}\bigr|}

\renewcommand{\bar}[1]{\overline{#1}}

\newcommand{\bigset}[2]{ \bigl\{ \, {#1} \bigm| {#2} \, \bigr\} }
\renewcommand{\emptyset}{\varnothing}

\renewcommand{\setminus}{-}

\newcommand{\field}[1]{\mathbb{#1}}
\newcommand{\Z}{\field{Z}}
\newcommand{\R}{\field{R}}

\newcommand{\N}{\field{N}}

\DeclareMathOperator{\CAT}{CAT}

\usepackage{ifthen}

\newcommand{\showcomments}{yes}

\newsavebox{\commentbox}
{\ifthenelse{\equal{\showcomments}{yes}}%
{\footnotemark
    \begin{lrbox}{\commentbox}
    \begin{minipage}[t]{1.25in}\raggedright\sffamily\upshape\tiny
    \footnotemark[\arabic{footnote}]}%
{\begin{lrbox}{\commentbox}}}%
{\ifthenelse{\equal{\showcomments}{yes}}%
{\end{minipage}\end{lrbox}\marginpar{\usebox{\commentbox}}}%
{\end{lrbox}}}

\begin{document}

\title{Subgroup distortion of 3-manifold groups}

\author{Hoang Thanh Nguyen}
\address{Beijing International Center for Mathematical Research\\
Peking University\\
 Beijing 100871, China
P.R.}
\email{htnguyen.dn.vn@outlook.com}
%\author{Hoang Thanh Nguyen} 
%\address{Department of Mathematical Sciences\\
%University of Wisconsin--Milwaukee\\
%P.O.~Box 413\\
%Milwaukee, WI 53201\\
%USA}
%\email{nguyen36@uwm.edu}

\author{Hongbin Sun}
\address{Department of Mathematics\\
Rutgers University-New Brunswick\\
Hill center, Busch Campus\\
Piscataway, NJ 08854\\
USA}
\email{hongbin.sun@rutgers.edu}

\date{\today}

\begin{abstract}
In this paper, we compute the subgroup distortion of all finitely generated subgroups of all finitely generated $3$--manifold groups, and the subgroup distortion in this case can only be linear, quadratic, exponential and double exponential. It turns out that the subgroup distortion of a subgroup of a $3$-manifold group is closely related to the separability of this subgroup.
\end{abstract}

%\subjclass[2000]{%
%20F67, % Hyperbolic groups and nonpositively curved groups
%20F65}  %Geometric group theory
%\keywords{graph manifold, mixed manifold, spirality, subgroup distortion}

\maketitle
\section{Introduction}
 A finitely generated group $G$ can be considered as a metric space when we equip the group with the word metric from a finite generating set. 
 Gromov has been successful in promoting this idea to study finitely generated
groups. With different finite generating sets on $G$, we have different metrics on $G$. However, such metric spaces are unique up to the quasi-isometric equivalence. Here, a map between two metric spaces is called a quasi-isometry if it is coarsely bi-Lipschitz and coarsely surjective.

A basic goal in geometric group theory that is proposed by Gromov is to quantifies the failure of the coarsely bi-Lipschitz property (called \emph{distortion}) of the inclusion $H \to G$ of
a finitely generated subgroup $H$ in a finitely generated group $G$.
 More precisely, let $\mathcal{S}$ and $\mathcal{A}$ be finite generating sets of $G$ and $H$ respectively. The \emph{distortion} of $H$ in $G$ is the function
\[
   \Delta_{H}^G(n)
     = \max \bigset{\abs{h}_{\mathcal{A}}}{\text{$h\in{H}$ and $\abs{h}_{\mathcal{S}}\leq{n}$} }
\]
Up to a natural equivalence, the function $\Delta_{H}^{G}$ does not depend on the choice of finite generating sets $\mathcal{S}$ and $\mathcal{A}$.  

In the 3-manifold theory, study fundamental groups of 3-manifolds is one of the most central topic, thus computing subgroup distortion in finitely generated 3-manifold groups is a natural task. In this paper, we consider the case that the group $G$ is an arbitrary finitely generated 3-manifold group, and compute subgroup distortion of all finitely generated subgroups of $G$.
%The purpose of this paper is to address the following natural problem:
\begin{prob}
\label{ques:naturalquestion}
Let $M$ be a 3-manifold with finitely generated fundamental group, and let $H$ be a finitely generated subgroup of $\pi_1(M)$. What is the distortion of $H$ in $\pi_1(M)$? How is this coarse geometric property related to algebraic properties of $H < \pi_1(M)$?
\end{prob}

The answer to Problem~\ref{ques:naturalquestion} is relatively well-understood when the manifold is geometric. If $M$ is a hyperbolic 3-manifold with empty or tori boundary, then any finitely generated subgroup $H$ is either linearly distorted or exponentially distorted in $\pi_1(M)$, by the Covering Theorem (see \cite{Can96}) and the Subgroup Tameness Theorem (see \cite{Agol04}, \cite{CG06}). If $M$ has a geometric structure modelled on $S^{3}$, $\R^{3}$ or $S^{2} \times \R$,  then $H$ is undistorted in $\pi_1(M)$ since $\pi_1(M)$ is virtually abelian. Subgroup distortion of the fundamental group of a 3-manifold supporting the Nil geometry is either linear or quadratic by the work of Osin (see \cite{Osin01}). Our contribution is to compute the subgroup distortion of a 3-manifold $M$ that has the geometry of Sol, $\mathbb{H}^{2} \times \R$ or $\widetilde{SL(2,\R)}$. These results are easy to prove, and may not be surprising to experts, but we can not find them in literature.

\begin{prop}[Subgroup distortion in geometric 3-manifolds]
\label{prop:distgeome}
Let $M$ be a geometric 3-manifold with empty or tori boundary, and let $H$ be a finitely generated subgroup of $\pi_1(M)$.  Let $\Delta$ be the distortion of $H$ in $\pi_1(M)$. Then the following statement holds:
\begin{enumerate}
    \item If the geometry of $M$ is either $S^3$, $\mathbb{E}^3$, $S^2 \times \R$, $\widetilde{SL(2,\R)}$ or $\mathbb{H}^{2} \times \R$ then $\Delta$ is linear.
    \item If the geometry of $M$ is Nil, $M$ is a Seifert $3$-manifold. Then $\Delta$ is quadratic if $H$ is an infinite index subgroup that intersects with the Seifert fiber subgroup of $\pi_1(M)$ nontrivially; otherwise $\Delta$ is linear.
    \item If the geometry of $M$ is $\mathbb{H}^3$, then $\Delta$ is linear if $H$ is geometrically finite, and is exponential if $H$ is a virtual fiber subgroup (i.e. geometrically infinite).
    \item If $M$ has the geometry of Sol, then $M$ is a torus bundle over a $1$-dimensional orbifold (an interval with reflection boundary points or a circle). Then $\Delta$ is exponential if $H$ is an infinite subgroup of the fiber subgroup of $\pi_1(M)$, otherwise it is linear.
    %\item If the geometry of $M$ is either $\widetilde{SL(2,\R)}$ or $\mathbb{H}^{2} \times \R$, then $\Delta$ is linear.
\end{enumerate}
\end{prop}

When $M$ is a non-geometric 3-manifold with empty or tori boundary, and the subgroup $H$ is associated to an immersed $\pi_1$--injective subsurface in $M$, i.e, $H = f_{*}(\pi_1(S))$ where $f \colon S \looparrowright M$ is a properly immersed $\pi_1$-injective surface, a complete computation of subgroup distortion is given by \cite{HN19} and \cite{Ngu18a}. In this case, the only possibility for the distortion is linear, quadratic, exponential and double exponential (see Theorem~1.2 in \cite{Ngu18a}).
So far, distortion of an arbitrary finitely generated subgroup of a finitely generated 3-manifold group is unknown before the project in this paper.

In this paper, we use the previous result in \cite{Ngu18a} as one of the key ingredients to generalize the main theorem in \cite{Ngu18a} to arbitrary finitely generated subgroups of finitely generated 3-manifold groups. We give a complete computation to the distortion of all finitely generated subgroups of all finitely generated 3-manifold groups, answering the first question in Problem~\ref{ques:naturalquestion}.

\begin{thm}\label{general}
Let $M$ be a $3$-manifold with finitely generated fundamental group and let $H < \pi_1(M)$ be a finitely generated subgroup. Then the distortion of $H$ in $\pi_1(M)$ can only be linear, quadratic, exponential and double exponential.
\end{thm}

%By Scott core Theorem, we can assume that the manifold $M$ is compact. By taking a double cover if necessary, we assume that $M$ is orientable. By Prime decomposition Theorem, 

By some standard arguments, we first reduce to the case where $M$ is compact, connected, orientable, irreducible and has empty or tori boundary (see the proof of Theorem~\ref{general}).
Theorem~\ref{general} actually follows from the following Theorem~\ref{toriboundary}, which not only shows what kind of subgroup distortion show up, but also provides a strong connection between subgroup distortion and separability of this subgroup. Here, a subgroup $H < G$ is called \emph{separable} if it is an intersection of finite index subgroups of $G$. The study of subgroup separability is very important in the 3-manifold theory and it is closely related with the virtual Haken conjecture that was resolved by Agol \cite{Agol13}. Recently, in \cite{Sun18}, the second author generalizes the work of Liu \cite{Liu17} and Rubinstein-Wang \cite{RW98} to give a complete characterization on which finitely
generated subgroups of finitely generated 3-manifold groups are separable. In particular, the second author introduces a notion called \emph{almost fiber surface} $\Phi(H)$ of a finitely generated subgroup in a 3-manifold group $H<\pi_1(M)$ (that is also a generalization of a notion called \emph{almost fiber part} in \cite{Liu17}), and show that all information
about separability of $H$ can be obtained by examining the almost fibered surface (see Theorem~1.3 in \cite{Sun18}). The almost fiber surface $\Phi(H)$ is naturally an embedded (possibly disconnected) subsurface of $M_H$, the covering space of $M$ corresponding to $H<\pi_1(M)$, and the torus decomposition of $M$ induces a decomposition of $\Phi(H)$. Under this decomposition of $\Phi(H)$, each of its component is called a \emph{vertex piece} or simply a \emph{piece} of $\Phi(H)$. More detail on the almost fiber surface will be provided in Section \ref{sec:almostfiber}, and we modify its definition a little bit so that it accommodates Theorem \ref{toriboundary}.

The following theorem answers the second question in Problem~\ref{ques:naturalquestion}. It shows that the subgroup distortion of a finitely generated subgroup of a 3-manifold group is closely related to the separability of this subgroup.

\begin{thm}\label{toriboundary}
Let $M$ be a compact orientable irreducible $3$-manifold with empty or tori boundary, with nontrivial torus decomposition and does not support the Sol geometry. Let $H<\pi_1(M)$ be a finitely generated subgroup, and let $\Phi(H)$ be the almost fiber surface of $H$. Let $\Delta$ be the distortion of $H$ in $\pi_1(M)$. There are four mutually exclusive cases:
\begin{enumerate}
    \item 
    \label{toriboundary:1}
    If there is a component $S$ of the almost fiber surface $\Phi(H)$ such that $S$ contains a geometrically infinite piece and $\pi_1(S)$ is non-separable in $\pi_1(M)$ then $\Delta$ is double exponential.
    \item 
    \label{toriboundary:2}
   Suppose that $\Phi(H)$ has no component satisfying (\ref{toriboundary:1}).  If there is a component $S$ of the almost fiber surface $\Phi(H)$ such that $S$ contains a geometrically infinite piece, then $\Delta$ is exponential.
    \item 
    \label{toriboundary:3}
    Suppose that $\Phi(H)$ has no component satisfying (\ref{toriboundary:1}) and (\ref{toriboundary:2}). If there is a component $S$ of the almost fiber surface $\phi(H)$ that contains at least two pieces, then $\Delta$ is exponential if $\pi_1(S)$ is non-separable in $\pi_1(M)$ and $\Delta$ is quadratic if $\pi_1(S)$ is separable in $\pi_1(M)$.
    \item 
    \label{toriboundary:4}
    In all other cases, $\Delta$ is linear.
\end{enumerate}
\end{thm}

We note that Theorem~\ref{toriboundary} generalizes Theorem~1.2 in \cite{Ngu18a} where the subgroup $H=f_{*}(\pi_1(S))$ for a clean surface $f \colon S \looparrowright M$. The strategy of the proof of Theorem~1.2 in \cite{Ngu18a} is the following. At first, the author shows that the distortion of the surface subgroup in the 3-manifold group is determined by the almost fiber surface. Then the author computes the distortion of each component of the almost fiber surface. We prove Theorem~\ref{toriboundary} by following the same strategy. The distortion of components of the almost fiber surface $\Phi(H)$ follows from the work in \cite{Ngu18a}, while showing the distortion of $H$ in $\pi_1(M)$ depends only on $\Phi(H)$ requires more work (see Theorem~\ref{technical}). This result is the technical heart of this paper.

%In the setting of a clean surface in a nongeometric 3-manifold with empty or tori boundary, the first author shows that in \cite{Ngu} that the distortion of the surface subgroup in the 3-manifold group is determined by the almost fiber surface. We are strengthen this result to a general setting that is finitely generated subgroups in 3-manifold groups (not neccessary surface subgroups). 
%Then we compute the distortion of components of the almost fibered surface $\Phi(H)$ in the 3-manifold. Meanwhile, the distortion of components of the almost fibered surface $\Phi(H)$ in the 3-manifold has been addressed in \cite{Ngu}, proving the subgroup distortion depends only on the almost fibered surface is nontrivial and is the heart of this paper.
\begin{thm}\label{technical}
Let $M$ be a $3$-manifold and let $H < \pi_1(M)$ be a subgroup as in Theorem \ref{toriboundary}. Let $\Phi(H)$ be the almost fiber surface of $H$ with connected components $S_1,\cdots,S_n$, and let $\delta_{S_i}$ be the distortion $\pi_1(S_i)$ in $\pi_1(M)$. Then the subgroup distortion of $H$ in $G=\pi_1(M)$ satisfies: 
$$\Delta^G_H\sim f.$$

Here $$f(n):=\max\{\delta_{S_i}(n)\ |\ S_i\text{\ is\ a\ component\ of\ }\Phi(H)\}.$$
\end{thm}

The proof of Theorem~\ref{technical} uses the same strategy as in \cite{Ngu18a}, however the techniques are very different. Unlike as in the setting of a properly immersed $\pi_1$--injective surface $S \looparrowright M$ where a compact surface $S$ is given, we only need to compute the distortion of $\tilde{S}$ in $\tilde{M}$. In the general setting, we need to consider the covering space $M_H$ of $M$ corresponding to $H < \pi_1(M)$, and then construct a Scott core $K \subset M_H$ with some special properties. We then compute the distortion of $\tilde{K}$ in $\tilde{M}$.

\subsection{Overview}
In Section~\ref{sec:preli} we review some concepts in geometric group
theory. Section~\ref{sec:almostfiber} is a review on the almost fiber surface of a finitely generated subgroup of a 3-manifold group. In Section~\ref{sec:techicalsection}, we give the proof of Theorem~\ref{technical}. In Section~\ref{sec:distortiongeometricmld}, we prove of Theorem~\ref{general} and Theorem~\ref{toriboundary}.

\subsection{Acknowledgements} 
The first author would like to thank Chris Hruska for helpful conversations. The second author is partially supported by NSF grant DMS-1840696. We thank Chris Hruska for his comments on a previous version of this preprint, and thank Martin Bridson for informing us an alternative proof of Proposition \ref{prop:SL2}. We are also grateful to the anonymous referee for many very helpful comments.

\bigskip
\bigskip

\section{Preliminaries}
\label{sec:preli}
In this section we review some concepts in geometric group theory that will be used in this paper.
\begin{defn}
\label{def:equivalentfunction}
Let $\mathcal{F}$ be the collection of all non-decreasing functions from positive reals to positive reals. Let $f$ and $g$ be arbitrary elements of $\mathcal{F}$. The function $f$ is \emph{dominated} by a function $g$, denoted by
\emph{$f\preceq g$}, if there are positive constants $A$, $B$, $C$, $D$ and $E$ such that
\[
  f(x) \leq A\,g(Bx+C)+Dx+E \quad \text{for all $x$.}
\]
Two functions $f$ and $g$ are \emph{equivalent},
denoted by $f\sim g$, if $f\preceq g$ and $g\preceq f$.
\end{defn}

\begin{rem}
The relation $\sim$ is an equivalence relation on the set $\mathcal{F}$. Let $f$ and $g$ be two polynomial functions with degree at least $1$ in $\mathcal{F}$, then it is not hard to show that they are equivalent if and only if they have the same degree. Moreover, all exponential functions of the form $a^{bx+c}$, where $a>1$, $b>0$ are equivalent.
\end{rem}

\begin{defn}[Subgroup distortion]
Let $H < {G}$ be a pair of finitely generated groups, and let $\mathcal{S}$ and $\mathcal{A}$ be finite generating sets of $G$ and $H$ respectively. The \emph{distortion} of $H$ in $G$ is the function
\[
   \Delta_{H}^G(n)
     = \max \bigset{\abs{h}_{\mathcal{A}}}{\text{$h\in{H}$ and $\abs{h}_{\mathcal{S}}\leq{n}$} }
\]
Up to equivalence, the function $\Delta_{H}^{G}$ does not depend on the choice of finite generating sets $\mathcal{S}$ and $\mathcal{A}$.
\end{defn}

\begin{lem}[Proposition~9.4 \cite{Hru10}]
\label{lem:HruskaProp9.4}
Let $G$ be a finitely generated group with a word length metric $d$. Suppose $H$ and $K$ are subgroups of $G$. For each constant $r$ there is a constant $r' = r'(G,d,H,K,r)$ so that in the metric space $(G,d)$ we have 
\[
\mathcal{N}_{r}(H) \cap \mathcal{N}_{r}(K) \subset \mathcal{N}_{r'}(H \cap K)
\]
\end{lem}

 \begin{defn}\label{def:superadditive}
A function $f \colon \N \to \N$ is \emph{superadditive} if 
\[
f(a+b) \ge f(a) + f(b) \, \, \, \text{for all $a,b \in \N$}
\]
The \emph{superadditive closure} of a function $f \colon \N \to \N$ is the function defined by the formula
\[
\bar{f}(n) = \max \bigset{f(n_1)+ \cdots + f(n_{\ell})}{\ell \ge 1 \text{\,\,and $n_1 +\cdots +n_{\ell} = n$}}
\]
\end{defn}

\begin{lem}[Proposition~2.5 in \cite{Ngu18a}]
\label{lem:preparationdist}
Let $K',K$ and $G'$ be finitely generated subgroups of a finitely generated group $G$ such that $K' < G'$ and $K' < K$. Suppose that $K'$ is undistorted in $K$ and $G'$ is undistorted in $G$, then $\Delta_{K'}^{G'} \preceq \Delta_{K}^{G}$.
\end{lem}
It is well known that a group acting properly, cocompactly, and isometrically on a geodesic space is quasi-isometric to the space.
The following corollary of this fact allows us to compute distortion using the geometries of spaces in place of word metrics.
\begin{cor}
\label{cor:GeometricDistortion}
Let $X$ and $Y$ be compact geodesic spaces, and let $g\colon{(Y,y_0)} \to (X,x_0)$ be $\pi_1$--injective. We lift the metrics on $X$ and $Y$ to geodesic metrics on universal covers $\tilde{X}$ and $\tilde{Y}$ respectively, with lifted map $\tilde{g}\colon (\tilde{Y},\tilde{y}_0)\to (\tilde{X},\tilde{x}_0)$. Let $G=\pi_1(X,x_0)$ and $H=g_{*} \bigl( \pi_1(Y,y_{0}) \bigr)$.
Then the distortion $\Delta^G_H$ is equivalent to the function
\[
   f(n) = \max \bigset{d_{\tilde Y}(\tilde{y}_0,h (\tilde{y}_0))}{\text{$h \in \pi_1(Y,y_0)$ and $d_{\tilde X}(\tilde{x}_0,g_*(h) (\tilde{x}_0)) \le n$}}.
\]
\end{cor}

\bigskip
\bigskip

\section{The almost fiber surface of a subgroup of a 3-manifold group}
\label{sec:almostfiber}
In this section, we briefly review an important notion in this paper that is called the \emph{almost fiber surface}, and we will do some modification on its original definition in \cite{Sun18}.

The almost fiber surface $\Phi(H)$ of a finitely generated subgroup $H$ of a finitely generated 3-manifold group is introduced by the second author in \cite{Sun18}. In \cite{Sun18}, the author gives a complete characterization on which finitely generated subgroups of finitely generated 3-manifold groups are separable and shows that all information about separability of the subgroup can be obtained by examining the almost fiber surface. In Section~\ref{sec:techicalsection}, we will show that the distortion of the subgroup in the 3-manifold group depends only on the almost fiber surface.

To get into the definition of the almost fiber surface, we need some terminology.

\begin{defn}[virtual fiber subgroup, partial fiber subgroup]
\label{defn:vf, partiallyfibered}
Let $M$ be a compact orientable irreducible 3-manifold with empty or tori boundary, and we assume that it is either a hyperbolic 3-manifold or a Seifert manifold. Let $H$ be a finitely generated subgroup of $\pi_1(M)$, and let $M_H \to M$ be the covering space corresponding to the subgroup $H$.
\begin{enumerate}
    \item When $M$ is a hyperbolic 3-manifold, $H$ is called a \emph{virtual fiber subgroup} if $H$ is geometrically infinite. In this case, $M_H$ is homeomorphic to $\Sigma_H \times \R$ or a twisted $\R$--bundle $\Sigma_H \tilde{\times} \R$ for a compact surface $\Sigma_H$.
    \item When $M$ is a Seifert manifold, $H$ is called a \emph{virtual fiber subgroup} if the induced bundle structure on $M_H$ is an $\R$--bundle and the base surface $\Sigma_H$ of $M_H$ is compact. In this case, $M_H$ is homeomorphic to $\Sigma_H \times \R$ or a twisted $\R$--bundle $\Sigma_H \tilde{\times} \R$. The subgroup $H$ is called a \emph{partial fiber subgroup} if the induced bundle structure on $M_H$ is an $\R$--bundle and the base surface $\Sigma_H$ of $M_H$ is noncompact. In this case, by abusing notation, there exists a compact subsurface $\Sigma_H \subset M_H$ such
that the inclusion induces an isomorphism on fundamental groups (see \cite{Sun18}). Moreover, we can assume that $\Sigma_H$ intersects with each cylinder boundary component of $M_H$ along a circle and does not interesct with any plane boundary component of $M_H$.
\end{enumerate}
\end{defn}

Let $M$ be a compact orientable irreducible $3$-manifold with empty or tori boundary, with nontrivial torus decomposition and does not support the Sol geometry. Let $H$ be a finitely generated infinite index subgoup of $G =\pi_1(M)$. We are going to define the almost fiber surface $\Phi(H)$.

Let $M_H \to M$ be the covering space corresponding to $H$. Since $M$ has nontrivial torus decomposition, $M_H$ has an induced graph of space structure. Each elevation (i.e. a component of the preimage) of a piece of $M$ in $M_H$ is called a \emph{vertex space} or a \emph{piece} of $M_H$, and each elevation of a decomposition torus of $M$ in $M_H$ is called an \emph{edge space} of $M_H$. We call a piece of $M_H$ a virtual fiber piece or a partial fiber piece if it covers the corresponding piece of $M$ by a way described in Definition \ref{defn:vf, partiallyfibered}. We denote
the dual graph of $M_H$ by $G_H$. Each vertex $v \in G_H$ corresponds to a piece in $M_H$, and we denote it by $M^{v}_H$.

Since $H$ is finitely generated, there exists a finite union of pieces $M_H^c\subset M_H$, such that the inclusion $M_H^c \to M_H$ induces an isomorphism on fundamental groups, and we take $M_H^c$ to be the minimal such manifold. Let $G^{c}_H$ be the subgraph of $G_H$ dual with $M^{c}_H$.

Let $J^{c}_H$ be the set of vertices $v \in G^{c}_H$  such that $M^{v}_H \subset M^{c}_H$ is a virtual fiber or 
partial fiber piece and has non-trivial fundamental group. 

\begin{defn}[The original definition of almost fiber surface in \cite{Sun18}]
For each $v \in J^{c}_H$, we take a copy $\Sigma_{H}^{v}$ that is given in Definition~\ref{defn:vf, partiallyfibered}. For each cylinder edge
space $C \subset M_H$ that intersects with the surfaces in $\{\Sigma^{v}_H \}_{v \in J^{c}_H}$ along exactly two
circles, we isotopy these two surfaces near the boundary, so that they intersect with $C$ along the same circle and we paste them together along the circle. After doing all these pasting, we get the \emph{almost fiber surface} $\Phi(H)$ and it is naturally a subsurface of $M^{c}_H$. Each $\Sigma_H^v$ is called a piece of $\Phi(H)$.
\end{defn}

\begin{rem}
 It is possible that the almost fibered surface $\Phi(H)$ is disconnected. 
\end{rem}

To accommodate the statement in Theorem \ref{toriboundary} (3), we modify the definition of the almost fiber surface as the following.
\begin{defn}\label{modified}[Modified definition of almost fiber surface]
In the almost fiber surface $\Phi(H)$, it has some piece (lying in partial fiber pieces) that is homeomorphic to the annulus and parallel to the boundary of $\Phi(H)$. We delete these annulus pieces from $\Phi(H)$ to get the modified almost fiber surface, and we still denote it by $\Phi(H)$.
\end{defn}

For each component $S_i\subset \Phi(H)$ under the origninal definition, either the modification deletes some annuli neighborhood of some boundary components of $S_i$, or $S_i$ is homeomorphic to a annulus (consists of one or two pieces) and we delete it from $\Phi(H)$. In the first case, it clearly does not change the fundamental group of $S_i$, and does not change $\Delta^{\pi_1(M)}_{\pi_1(S_i)}$, but it affects the statement in Theorem \ref{toriboundary} (3). In the second case, the subgroup distortion $\Delta^{\pi_1(M)}_{\pi_1(S_i)}$ is linear (see Remark \ref{rem:vertexandedgegroup}).

\bigskip
\bigskip

\section{Distortion of a finitely generated subgroup is determined by the almost fiber part}
\label{sec:techicalsection}

In this section we prove the technical heart of this paper: Theorem \ref{technical}. Actually, we will prove that $$f\preceq \Delta^G_H\preceq \bar{f}$$ holds, here $\bar{f}$ is the superadditive closure of $f$ as in Definition \ref{def:superadditive}. Since surface subgroups of graph of mixed 3-manifold groups can only have linear, quadratic, exponential or double exponential distortions (by \cite{Ngu18a}), $\bar{f}\sim f$ holds, and Theorem \ref{technical} follows from the above inequality.

The proofs of the lower bound part and the upper bound part  of are given in Subsection~\ref{lowerbound} and Subsection~\ref{upperbound} respectively.

The proof will follow the same strategy as in \cite{Ngu18a}, i.e. apply Corollary~\ref{cor:GeometricDistortion}. More precisely, we first take a preferred Riemannian metric on $M$, and consider the covering space $M_H$ of $M$ corresponding to $H < \pi_1(M)$. Then we construct a Scott core $K\subset M_H$ (by \cite{Sco73}, here Scott core means that $K$ is a compact codimension-$0$ submanifold of $M_H$ such that the inclusion is a homotopy equivalence) with some nice property, take the universal cover $\tilde{M}$ of $M$ and take the preimage of $K$ in $\tilde{M}$ to get $\tilde{K}\subset \tilde{M}$. We lift the Riemannian metric to $\tilde{M}$ and take the induced path metric on $\tilde{M}$ and $\tilde{K}$ and denote them by $d_{\tilde{M}}$ and $d_{\tilde{K}}$ respectively. Then by Lemma \ref{cor:GeometricDistortion}, we have $$\max{\{d_{\tilde{K}}(x,y)\ |\ x,y\in \tilde{K} \text{\ and\ } d_{\tilde{M}}(x,y)\leq n\}} \sim \Delta^G_H.$$ To prove Theorem \ref{technical}, it suffices to prove that: $$f\preceq \max{\{d_{\tilde{K}}(x,y)\ |\ x,y\in \tilde{K} \text{\ and\ } d_{\tilde{M}}(x,y)\leq n\}}\preceq \bar{f}.$$

\subsection{Some preparation}\label{preparation}

At first, by passing to a finite cover $M'$ of $M$, we can assume that each Seifert piece $M_i\subset M$ is a product $F_i \times S^1$, and $M$ does not contain the twisted $I$-bundle over the Klein bottle (see Lemma~3.1 in \cite{PW14}). Then we will study the subgroup distortion of $H'=H\cap \pi_1(M')<\pi_1(M')$. Since $H'$ is a finite index subgroup of $H$, we have $\Delta^{\pi_1(M')}_{H'} \sim \Delta^{\pi_1(M)}_H$ and $\Phi(H')$ is a finite cover of $\Phi(H)$ (Lemma 3.6 of \cite{Sun18}). So it suffices to prove Theorem \ref{technical} for $H'<\pi_1(M')$, and we still denote the subgroup of the $3$-manifold group by $H<\pi_1(M)$.

{\bf Preparation Step I: A metric on $M$.}
\label{preparationstepI}
 Since the choice of length metrics does not affect the distortion, we will choose a convenient metric on $M$ as the following.

If $M$ is a mixed manifold, we take a nonpositive curved Riemannian metric $d$ on $M$ as in \cite{Leeb95}. In this case, each decomposition torus in $M$ is a totally geodesic subsurface, and the restriction of $d$ on each decomposition torus of $M$ is a flat metric.

If $M$ is a graph manifold, we take a Riemannian metric $d$ (may not be nonpositive curved) on $M$ such that each decomposition torus is a totally geodesic subsurface, and the restriction of $d$ on such a torus is a flat metric. 

The construction of the metric $d$ is given as the following. For each Seifert piece $M_i=F_i\times S^1$ of $M$, we fix a metric $d_i$ on $M_i$ (say a hyperbolic metric on the surface with geodesic boundary cross the standard metric on the circle); for each decomposition torus $T$, we fix a flat metric $d_T$ on it. Take a Seifert piece $M_i$, for simplicity, we assume that it is adjacent to a single decomposition torus $T$. we take a collar of $T$ in $M_i$ to get $T\times [0,1]\subset M_i$, with $T\subset \partial M_i$ being identified with $T\times \{1\}$, and let $M_i^{\text{int}}$ be the closure of the complement of $T\times [0,1]$ in $M_i$. Then we have $M_i=M_i^{\text{int}}\cup (T\times [0,1])$. Let $U$ be the open subset $M_i^{\text{int}}\cup (T\times [0,1/2))$ and let $V$ be the open subset $T\times (0,1]$. We take the restriction of the metric $d_i$ on $U$, and take the restriction of the metric $d_T\times d_{\text{Euc}}$ on $V$ (here $d_{\text{Euc}}$ denotes the standard metric on the interval). Then the partition of unity gives a metric on $M_i$, such that its restriction on a neighborhood of $T$ is $d_T\times d_{\text{Euc}}$. Under this metric, $T$ is a totally geodesic subsurface and the restriction on $T$ is flat. Since each decomposition torus $T$ is given a fixed flat metric $d_T$, we can paste the metrics on these pieces $M_i\subset M$ to a Riemannian metric on $M$, which is our desired metric $d$.

%If $M$ is a graph manifold, we take the Riemannian metric $d$ on $M$ as in Remark 4.13 of \cite{HN}. If $M$ is a mixed manifold, we take a nonpositive curved Riemannian metric $d$ on $M$ as in \cite{Ngu}. In either case, each decomposition torus in $M$ is a totally geodesic subsurface, and the restriction of $d$ on each decomposition torus of $M$ is a flat metric.

We can take a product metric $d_i$ on each Seifert piece $M_i=F_i\times S^1 \subset M$ as above, and take a truncated hyperbolic metric $d_i$ on each hyperbolic piece $M_i\subset M$ (they may not be the restriction of the metric $d$).
For the metric $d$ on $M$ and the metrics $d_i$ on pieces of $M$, $d|_{M_i}$ and $d_i$ are bilipschitz to each other. By Lemma 1.8 of \cite{Pau05}, there exists a constant $\kappa>1$, such that for any two points $x,y\in \tilde{M}$ lying in the same piece $\tilde{M}_i$, we have $$\frac{1}{\kappa}d_i(x,y)\leq d(x,y)\leq \kappa d_i(x,y).$$ Moreover, we can also assume that for any path $\gamma$ in $M_i$, $$\frac{1}{\kappa}|\gamma|_{d_i}\leq |\gamma|\leq \kappa |\gamma|_{d_i}$$ holds. Here $|\gamma|$ denotes the length of $\gamma$ under the $d$-metric.

{\bf Preparation Step II: A Scott core of $M_H$.}
Let $M_H$ be the covering space of $M$ corresponding to $H < \pi_1(M)$, then it has an induced graph of space structure. Since $H$ is finitely generated, there exists a finite union of pieces $M_H^c\subset M_H$, such that the inclusion $M_H^c \subset M_H$ induces an isomorphism on $\pi_1$, and $M_H^c$ is the minimal such manifold.

For each piece $M_{H,i}\subset M_H^c$, we take a (compact) Scott core $K_i\subset M_{H,i}$ (\cite{Sco73}) such that the following holds:
\begin{enumerate}
  \item For each component of $E\subset \partial (M_{H,i})\cap \partial (M_H^c)$, $K_i\cap E$ is empty.
  \item For each torus component $T$ of $\partial (M_{H,i})$ that is contained in $\text{int}(M_H^c)$, $T$ is contained in $K_i$.
  \item For each cylinder component $C$ of $\partial (M_{H,i})$ that is contained in $\text{int}(M_H^c)$, $C\cap K_i$ is a convex annulus in $C$. More precisely, since the restriction metric of $d$ on each decomposition torus is flat, the metric on $C$ is a geometric product $S^1\times \mathbb{R}$ (up to rescaling the metric on $S^1$), thus we have $C\cap K_i=S^1\times [a_i,b_i]$.
  \item For each plane component $P$ of $\partial (M_{H,i})$ that is contained in $\text{int}(M_H^c)$, $P\cap K_i$ is a convex set in the Euclidean plane $P$.
\end{enumerate}

Then we union these $K_i$'s together in $M_H^c$ to get a compact set $K'$. For each edge space $E\subset \text{int}(M_H^c)$, let $K_i$ and $K_j$ be the Scott cores of the pieces of $M_H^c$ adjacent to $E$ (it is possible that $K_i=K_j$). If $E$ is a torus, then $K_i\cap E=K_j\cap E=E$ holds and we do no further modification. If $E$ is a cylinder or a plane, we add the convex closure of $(K_i\cap E)\cup (K_j\cap E)$ in $E$ to $K$. (For example, in the cylinder case, we have $K_i\cap E=S^1\times [a_i,b_i]$ and $K_j\cap E=S^1\times [a_j,b_j]$, then we add $S^1\times [\min{\{a_i,a_j\}},\max{\{b_i,b_j\}}]$ to $K'$.) Let $K$ be an $\epsilon$-neighborhood of the above expansion of $K'$, then it is a Scott core of $M_H$. The important features of $K$ are:
\begin{enumerate}
  \item For each edge space $E\subset M_H$, $K\cap E$ is either empty or a convex subset of $E$ (under the restriction of both $d$ and $d_i$ metric).
  \item Since $K$ is compact, there exists $D>0$, such that for and edge space $E\subset M_H$ that intersects with $K$ and any two points $x,y\in K\cap E$, $d_E(x,y)<D$ holds.
\end{enumerate}

\begin{rem}
  Actually, for a virtual fiber or partial fiber piece $M_{H,i}\subset M_H^c$, we can assume that $K_i$ is a surface cross interval. We can also assume that $\Phi(H)$ is contained in $K$.
\end{rem}

\bigskip

\subsection{Two lemmas on metric properties of geometric pieces}\label{subsec:lemmas}

To prove Theorem~\ref{technical}, we need two lemmas.

The first lemma is parallel to Lemmas 4.4 and 4.6 of \cite{Ngu18a}, which describes the metric property of the preimage of $K_i\subset M_{H,i}$ in its universal cover, in the case that $M_{H,i}$ is either an $S^1$-bundle piece (that covers a Seifert piece of $M$) or a geometrically finite piece (that covers a hyperbolic piece of $M$).

\begin{lem}\label{easypiece}
  Let $M_{H,i}$ be an $S^1$-bundle piece or a geometrically finite piece of $M_H^c$, let $\tilde{M}_i$ be its universal cover, and let $\tilde{K}_i$ be the preimage of $K_i$ in $\tilde{M}_i$. Then there exists a constant $R_i$ depending only on $M_{H,i}$ and $K_i$ such that the following holds.

  Let $E$ and $E'$ be two boundary components of $\tilde{M}_i$ that intersect with $\tilde{K}_i$, and let $x,y$ be two points in $E$ and $E'$ respectively. Then there exists a path $\alpha$ in $\tilde{M}_i$ connecting $x$ and $y$, and a path $\beta$ in $\tilde{K}_i$ connecting some point $x'\in E\cap \tilde{K}_i$ and $y'\in E'\cap \tilde{K}_i$, such that the following holds.
  \begin{enumerate}
    \item Both $x'$ and $y'$  lie in the $R_i$-neighborhood of $\alpha$ (under $d_i$-metric).
    \item $|\alpha|_{d_i} = d_i(x,y)$.
    \item $|\beta|_{d_i} < R_id_i(x',y')$.
  \end{enumerate}
\end{lem}
\begin{proof}
Let $M_i$ be the piece of $M$ covered by $M_{H,i}$. Now we consider two cases:

{\bf{Case 1:}} $M_{H,i}$ is a geometrically finite piece of $M^{c}_H$.
%{\bf{We first prove the lemma in the case $(M_H)_i$ is a geometrically finite piece of $M^{c}_H$}}. 

We recall that $K_i$ is a compact Scott core of $M_{H,i}$.
By our assumption, $\pi_1(K_i)$ is a geometrically finite subgroup of $\pi_1(M_i)$. We note that $(\tilde{M}_i, d_i)$ is a $CAT(0)$ space. By Corollary~1.6 in \cite{Hru10}, the orbit space $\pi_1(K_i) \cdot \tilde{x}_{0}$ is quasiconvex in $(\tilde{M}_i,d_i)$ (in the sense that there exists a constant $k$ such that every geodesic of $\tilde{M}_i$ connecting two points of $\pi_1(K_i) \cdot \tilde{x}_0$ lies in the $k$--neighborhood of $\pi_1(K_i) \cdot \tilde{x}_0$). Thus, there exists a constant $\epsilon_0$ such that $\tilde{K}_i$ is $\epsilon_0$--quasiconvex in $\tilde{M}_i$.

By applying Lemma~\ref{lem:HruskaProp9.4} to $\pi_1(K_i)$ and the fundamental group of each torus boundary of $M_i$, we have the following fact: For any $r>0$, there exists $r' =r'(r)>0$ such that whenever $\tilde{T}$ is an arbitrary boundary plane of $\tilde{M}$ with nonempty intersection with $\tilde{K}_i$ and $x \in \mathcal{N}_{r}(\tilde{T}) \cap \mathcal{N}_{r}(\tilde{K}_i)$ , then $x \in \mathcal{N}_{r'}(\tilde{T} \cap \tilde{K}_i)$. Here we use that $K_i$ intersects with only finitely many boundary components of $M_{H,i}$.

We note that $(\tilde{M}_i,d_i)$ is a $\CAT(0)$ space with isolated flats. Let $\epsilon_{1}$ be the positive constant given by Proposition~8 \cite{HK09}. Let $[p,q]$ be a geodesic of shortest length from $E$ to $E'$. Then every geodesic from $E$ to $E'$ must intersect with the $\epsilon_{1}$-neiborhoods of both $p$ and $q$.

Let $\alpha$ be a geodesic in $(\tilde{M}_i,d_i)$ connecting $x \in E$ to $y \in E'$, it follows that $\{p,q\} \in \mathcal{N}_{\epsilon_1}(\alpha)$. Since $\alpha$ is a geodesic, (2) is confirmed.
We are going to establish (1). We note that $E \cap \tilde{K}_i \neq \emptyset$ and  $E' \cap \tilde{K}_i \neq \emptyset$. We choose a point in $E \cap \tilde{K}_i \neq \emptyset$ and choose a point in $E' \cap \tilde{K}_i \neq \emptyset$, and let $\gamma$ be a geodesic connecting these two points. It follows that $\{p,q\} \in \mathcal{N}_{\epsilon_1}(\gamma)$. Thus, there exist points $a$ and $b$ in $\gamma$ such that $d_{i}(a,p) \le \epsilon_{1}$ and $d_{i}(b,q) \le \epsilon_{1}$. Hence $a \in \mathcal{N}_{\epsilon_1}(E)$ and $b \in \mathcal{N}_{\epsilon_1}(E')$. We note that the end points of $\gamma$ belong to $\tilde{K}_i$. Using quasiconvexity of $\tilde{K}_i$, we have $a,b \in \mathcal{N}_{\epsilon_0}(\tilde{K}_i)$. Thus there exists a constant $\epsilon_{2}$ depending on $\epsilon_{0}$ and $\epsilon_{1}$ such that $a \in \mathcal{N}_{\epsilon_2}(E) \cap \mathcal{N}_{\epsilon_2}(\tilde{K}_i)$ and $b \in \mathcal{N}_{\epsilon_2}(E') \cap \mathcal{N}_{\epsilon_2}(\tilde{K}_i)$ (we may choose $\epsilon_2 = \epsilon_0 + \epsilon_1$).
Let $r' = r'(\epsilon_2)$ be the constant given by Lemma \ref{lem:HruskaProp9.4}, with respect to $\epsilon_2$. It follows that $a \in \mathcal{N}_{r'}(E \cap \tilde{K}_i)$ and $b \in \mathcal{N}_{r'}(E' \cap \tilde{K}_i)$. Thus, $d_{i}(a,x') \le r'$ and $d_{i}(b,y') \le r'$ for some points: $x' \in E \cap \tilde{K}_i$ and $y' \in E' \cap \tilde{K}_i$. Let $\beta$ be a shortest length in $\tilde{K}_i$ connecting $x'$ to $y'$. Since $d_{i}(\beta(0),p) = d_{i}(x',p) \le d_{i}(x',a) +d_{i}(a,p) \le r'+\epsilon_1$ and $p \in \mathcal{N}_{\epsilon_1}(\alpha)$, it follows that $\beta(0) \in \mathcal{N}_{r'+2\epsilon_1}(\alpha)$. Similarly, since  $d_{i}(\beta(1),q) = d_{i}(y',q) \le d_{i}(y',b) +d_{i}(b,q) \le r'+\epsilon_1$ and $q \in \mathcal{N}_{\epsilon_1}(\alpha)$, it follows that $\beta(1) \in \mathcal{N}_{r'+2\epsilon_1}(\alpha)$. Let $R_i =r'+2\epsilon_1$, then item (1) is verified.

We are going to verify (3). 
Since $\tilde{K}_i$ is undistorted in $\tilde{M}_i$, there exists a constant $R>0$ such that for any $a, b \in \tilde{M}_i$, $d_{\tilde{K}_i}(a,b) \le Rd_{i}(a,b) + R$ holds. Let $\rho$ be the lower bound of the $d_i$--distance for any pair of boundary planes of $\tilde{M}_i$. We have that $\abs{\beta}_{d_i} = d_{\tilde{K}_i}(x',y') \le R\,d_{i}(x',y') + R \le R\,d_{i}(x',y') + \frac{R}{\rho}d_{i}(x',y') = (R + \frac{R}{\rho})d_{i}(x',y')$. We may need to enlarge the constant $R_i$ to make sure that $R_i$ is bigger than $R+ \frac{R}{\rho}$. 

{\bf{Case 2:}} $M_{H,i}$ is a $S^1$--bundle piece.

 In this case, we recall  that $M_i = F_{i} \times S^1$ where $F_i$ is a hyperbolic surface with boundary. Then we have $\tilde{M}_i=\tilde{F}_i\times \mathbb{R}$ and we identify $\tilde{F}_i$ with $\tilde{F}_i\times \{0\}\subset \tilde{M}_i$.
We state here some facts that will be used in the rest of the proof.

{\bf{Fact~1:}}
$(\tilde{F}_i, d_{\tilde{F}_i})$ is bilipschitz homeomorphic to a fattened
tree (see the paragraph after Lemma~1.1 in \cite{BN08}). Thus, there exists $A_{0} >0$ such that the following holds. Let $\ell$ and $\ell'$ be two distinct boundary lines in $\tilde{F}_i$. Let $[p,p']$ be a geodesic of shortest length from $\ell$ to $\ell'$. If $\tau$ is a path in $\tilde{F}_i$ connecting a point in $\ell$ to a point in $\ell'$ then $[p,p'] \subset \mathcal{N}_{A_0}(\tau)$ where $\mathcal{N}_{A_0}(\tau)$ is the $A_0$--neighborhood of $\tau$ with respect to the $d_{\tilde{F}_i}$--metric.

{\bf{Fact~2:}} Let $proj(K_i)$ be the projection of $K_i$ into the base surface $F_i$ of $M_i$ (under the composition $M_{H,i} \to M_{i} \to F_{i}$).
By applying Lemma~\ref{lem:HruskaProp9.4} to $proj_*(\pi_1(K_i))$ and fundamental groups of boundary circles of $F_{i}$, and using the fact $\tilde{F}_i$ is a fattened tree, we have the following: There exists a constant $\delta >0$ such that the following holds.  Let $E$ and $E'$ be any two distinct boundary planes of $\tilde{M}_i$ such that they have non-empty intersection with $\tilde{K}_i$. Let $\ell$ and $\ell'$ be two boundary components of $\tilde{F}_i$ such that $\ell \subset E$ and $\ell' \subset E'$. Let $[p,p']$ be a geodesic of shortest length from $\ell$ to $\ell'$. Then $d( p,\ell \cap \tilde{K}_i) \le \delta$ and $d(p', \ell' \cap \tilde{K}_i) \le \delta$. (Its proof is similar to the previous geometrically finite case.)

Since $M_{H,i}$ is an $S^1$--bundle, it follows that $\partial (M_{H,i})$ consists of only tori and cylinders. Let $T$ be a torus component of $\partial (M_{H,i})$. If $T$ is a component of $\partial (M^{c}_H)$ then $T \cap K_i = \emptyset$; if $T$ is not a component of $\partial (M^{c}_H)$, then $T \cap K_{i} =T$. Let $C$ be a cylinder component of $\partial (M_{H,i})$. If $C$ is a component of $\partial (M^{c}_H)$, then $C \cap K_i = \emptyset$; if $C$ is not a component of $\partial (M^{c}_H)$, then $C \cap K_{i} = S^{1} \times [a_i,b_i]$, by the convexity of $K_i\cap \partial (M_{H,i})$.

% Let $T$ be the boundary component of $M_{H,i}$ such that $E$ universally covers $T$. Let $T'$ be the boundary component of $M_{H,i}$ such that $E'$ universally covers $T'$. 

We are now going to construct a path $\beta$ satisfying (1). We will use Fact~1 and Fact~2 here. %We consider the following cases:
Let $\alpha$ be a geodesic in $(\tilde{M}_i,d_i)$ connecting $x \in E$ to $y \in E'$.
Let $\alpha_{\tilde{F}_i}$ be the projection of $\alpha$ on the first factor $\tilde{F}_i$ of $\tilde{M}_i$. Let $\ell_0$ and $\ell_1$ be the boundary components of $\tilde{F}_i$ such that $\ell_0 \subset E$ and $\ell_1 \subset E'$ hold respectively. Let $[p_0,p_1]$ be a geodesic of shortest length from $\ell_0$ to $\ell_1$. According to the Fact~1, we have $[p_0,p_1] \subset \mathcal{N}_{A_0}(\alpha_{\tilde{F}_i})$. It follows that there exist $a,b \in \alpha_{\tilde{F}_i}$ such that $d_{\tilde{F}_i}(p_0,a) \le A_0$ and $d_{\tilde{F}_i}(p_1,b) \le A_0$. Using Fact~2, there exist points $u_0 \in \ell_0 \cap \tilde{K}_i$ and $u_1 \in \ell_1 \cap \tilde{K}_i$ such that $d_{i}(p_0, u_0) \le \delta$ and $d_{i}(p_1,u_1) \le \delta$. Thus,
\[
d_{i}(u_0, a) \le d_{i}(u_0, p_0) + d_{i}(p_0, a) \le \delta + A_0
\] and
\[
d_{i}(u_1, b) \le d_{i}(u_1, p_1) + d_{i}(p_1, b) \le \delta + A_0
\]
We choose $s_0 \in \R$ and $s_1 \in \R$ such that $(a, s_0) \in \alpha$ and $(b, s_1) \in \alpha$. Since $E\cap \tilde{K}_i$ and $E'\cap \tilde{K}_i$ are both union of $\mathbb{R}$-fibers, we have that $(u_0,s_0) \in E \cap \tilde{K}_i$ and $(u_1,s_1) \in E' \cap \tilde{K}_i$, while
\[
d_{i}\bigl ( (u_0,s_0), \alpha \bigr) \le d_{i} \bigl ((u_0,s_0), (a, s_0) \bigr ) \le d_{\tilde{F}_{i}}(u_0,a) \le \delta + A_0
\]
and 
\[
d_{i}\bigl ( (u_1,s_1), \alpha \bigr) \le d_{i} \bigl ((u_1,s_1), (b, s_1) \bigr ) \le d_{\tilde{F}_{i}}(u_1,b) \le  \delta + A_0.
\]
Let $\beta$ be a shortest path in $\tilde{K}_i$ connecting $x'=(u_0, s_0)$ to $y'=(u_1,s_1)$. If we choose $R_i > \delta + A_0$, it is easy to see that $\beta$ satisfies (1). The path $\alpha$ satisfies (2) since it is a geodesic in $(\tilde{M}_i,d_i)$. For (3), the proof is done by following the same argument as in the last paragraph in the proof of the lemma in geometrically finite case.
\end{proof}

The following lemma describes the metric property of the preimage of $K_i\subset M_{H,i}$ in its universal cover, in the case that $M_{H,i}$ is a partial fiber piece (that covers a Seifert piece of $M$).

\begin{lem}\label{hardpiece}
  Let $M_{H,i}$ be a partial fiber piece of $M_H^c$, let $\tilde{M}_i$ be its universal cover, and let $\tilde{K}_i$ be the preimage of $K_i$ in $\tilde{M}_i$. Then there exists a constant $R_i$ depending only on $M_{H,i}$ and $K_i$ such that the following holds.

  Let $E$ and $E'$ be two boundary components of $\tilde{M}_i$ that intersect with $\tilde{K}_i$, and let $x,y$ be two points in $E$ and $E'$ respectively. Then there exists a path $\alpha$ in $\tilde{M}_i$ connecting $x$ and $y$, and a path $\beta$ connecting some points $x'\in E$ and $y'\in E'$, such that the following holds.
  \begin{enumerate}
    \item The projection of $x'$ and $y'$ in the base surface of $\tilde{M}_i$ lie in the projection of $\tilde{K}_i$. 
    \item Both $x'$ and $y'$ are contained in the $R_i$-neighborhood of $\alpha$ (under $d_i$ metric).
    \item $|\alpha|_{d_i}=d_i(x,y)$.
    \item $|\beta|_{d_i}<R_id_i(x',y')$.
  \end{enumerate}

\end{lem}

\begin{rem}
  Note that we do not request that $\beta$ lies in $\tilde{K}_i$, which is different from Lemma \ref{easypiece}.
\end{rem}
\begin{proof}
 Since $M_{H,i}$ is a partial fiber piece, it follows that $\partial (M_{H,i})$ consists of only finitely many cylinder components and (possibly) infinitely many plane components.

Let $\alpha$ be a geodesic in $(\tilde{M}_i, d_i)$ connecting $x$ to $y$. Then $\alpha$ satisfies (3).
Let $\alpha_{\tilde{F}_i}$ be the projection of $\alpha$ in the base surfaces $\tilde{F}_i$ of $\tilde{M}_i$. Let $\ell_0$ and $\ell_1$ be the boundary lines of $\tilde{F}_i$ that is contained in $E$ and $E'$ respectively. Let $[p_0,p_1]$ be a geodesic of shortest length from $\ell_0$ to $\ell_1$. Using Fact~1, we have that $[p_0,p_1] \subset \mathcal{N}_{A_0}(\alpha_{\tilde{F}_i})$. 

 Since we assume that $E \cap \tilde{K}_i \neq \emptyset$ and $E ' \cap \tilde{K}_i \neq \emptyset$, it follows that $\ell \cap proj(\tilde{K}_i) \neq \emptyset$ and $\ell' \cap proj(\tilde{K}_i) \neq \emptyset$. Here $proj(\tilde{K}_i)$ is the projection of $\tilde{K}_i$ in the base surface $\tilde{F}_i$ of $\tilde{M}_i$.

By Fact 2, there are points $u_0 \in \ell \cap proj(\tilde{K}_i)$ and $u_1 \in \ell' \cap proj(\tilde{K}_i)$, such that $d(u_0, p_0) \le \delta$ and $d(u_1, p_1) \le \delta$. Since $[p_0,p_1] \subset \mathcal{N}_{A_0}(\alpha_{\tilde{F}_i})$, it follows that there exist some point $a \in \alpha_{\tilde{F}_i}$ and $b \in \alpha_{\tilde{F}_i}$ such that $d_{i}(p_0, a) \le A_0$ and $d_{i}(p_1, b) \le A_0$. 
Since $\alpha_{\tilde{F}_i}$ is the projection of $\alpha$ to $\tilde{F}_i$ of $\tilde{M}_i = \tilde{F}_i \times \R$, we choose $s_0$ and $s_1$ in $\R$ such that $(a,s_0) \in \alpha$ and $(b, s_1) \in \alpha$. Let $x' = (u_0, s_0)$ and $y' =(u_1,s_1)$. We note that $x' \in E$ and $y' \in E'$ and we have
\[
d_{i}\bigl ( x', (a,s_0) \bigr ) = d_{i}\bigl ( (u_0,s_0), (a,s_0) \bigr ) = d_{i}(u_0,a) \le d_{i}(u_0, p_0) + d_{i}(p_0, a) \le \delta + A_0
\] and
\[
d_{i}\bigl (y', (b,s_1) \bigr ) = 
d_{i}\bigl ( (u_1,s_1), (b,s_1) \bigr ) = d_{i}(u_1,b) \le d_{i}(u_1, p_1) + d_{i}(p_1, b) \le \delta + A_0
\]
Thus, $x'$ and $y'$ lie in the $\delta + A_0$--neighborhood of $\alpha$ (w.r.t $d_i$--metric), and we take an $R_i$ greater than $\delta+A_0$.

We are going to construct a path $\beta$ connecting $x'$ to $y'$. Let $\gamma \colon [0,1] \to proj(\tilde{K}_i)$ be a shortest path in $proj(\tilde{K}_i) \subset \tilde{F}_i$ connecting $u_0$ to $u_1$. Let $\beta \colon [0,1] \to \tilde{M}_i$ be defined by $\beta(t) = (\gamma(t), (1-t)s_0 + ts_1)$. Then the choice of $x'$ and $y'$ implies that (1) and (2) of the lemma hold. The path $\beta$ also satisfies (4) by using a similar argument as in the proof of geometrically finite case.
\end{proof}

In the rest of this section, we are going to prove Theorem~\ref{technical}. The proof of the lower bound part is given in Section~\ref{lowerbound}, and the proof of the upper bound part is given in Section~\ref{upperbound}.

\bigskip

\subsection{Lower bound of subgroup distortion}\label{lowerbound}
In this subsection, we are going to prove $f \preceq \Delta_{H}^{G}$ where $f$ is the function defined in Theorem~\ref{technical}. We need the following lemma.
\begin{lem}
\label{lem:undistorted}
Let $S_i$ be a connected component of the almost fiber surface $\Phi(H)$. Then  $\pi_1(S_i,s_0)$ is undistorted in $H$.
\end{lem}
\begin{rem}
When $H$ is a surface subgroup then Lemma~\ref{lem:undistorted} is obvious since every finitely generated subgroup of a surface group is quasiconvex.
\end{rem}
\begin{proof}[Proof of Lemma~\ref{lem:undistorted}]
  We take the pieces of the Scott core $K$ that intersect with $S_i$ nontrivially, and paste these pieces only along annuli but not discs. We denote the resulting submanifold of $K$ by $K_{S_i}$. Then the inclusion $S_i\subset K_{S_i}$ is a homotopy equivalence.
  
  Let $\tilde{M}$ be the universal cover of $M$, let $\tilde{K}$ be the preimage of $K$ in $\tilde{M}$, and let $\tilde{K}_{S_i}$ be one elevation of $K_{S_i}$ contained in $\tilde{K}$. We take a basepoint $s_0\in K_{S_i}$. Let $d$ be the Riemannian metric on $M$ given in Preparation Step I, and we lift this metric to $\tilde{M}$ and denote it by $d_{\tilde{M}}$. We denote the induced path metrics on $\tilde{K}$ and $\tilde{K}_{S_i}$  by $d_{\tilde{K}}$ and $d_{\tilde{K}_{S_i}}$ respectively. To see that $\pi_1(S_i, s_0)$ is undistorted in $H$, it suffices to show that $(\tilde{K}_{S_i}, d_{\tilde{K}_{S_i}})$ is undistorted in $(\tilde{K}, d_{\tilde{K}})$. 
  
 {\bf Claim}: Let $\kappa$ be the constant given by Preparation Step I. Then \\
 $d_{\tilde{K}_{S_i}}(\tilde{s}_0, h(\tilde{s}_0)) \le \kappa^2 \,d_{\tilde{K}}(\tilde{s}_0, h(\tilde{s}_0))$ for all $h \in \pi_1(S_i, s_0)$.

Let $\gamma$ be a path of shortest length in $(\tilde{K}, d_{\tilde{K}})$ connecting $\tilde{s}_0$ to $h(\tilde{s}_0)$. Let $\alpha$ be a maximum subpath of $\gamma$ not lying in $\tilde{K}_{S_i}$. Since the dual graph of $\tilde{M}$ is a tree and $K_{S_i}$ is the union of pieces of $K$ that contains pieces of $S_i$, it follows that $\alpha(0)$ and $\alpha(1)$ belong to the same decomposition plane $P$ of $\tilde{M}$. Let $\beta_{\alpha}$ be the geodesic path in the plane $(P,d_P)$ connecting $\alpha(0)$ to $\alpha(1)$. According to the Preparation Step I, it follows that $\beta_{\alpha}$ is also a geodesic on $P$ under the restriction of the metric $d_i$ on any piece $\tilde{M}_i\subset \tilde{M}$ that contains $P$. Since $d_i$ is a (3-dimensional) truncated hyperbolic metric or a (2-dimensional) hyperbolic metric cross the circle, we have
$$|\beta_{\alpha}|_{d} \le \kappa |\beta_{\alpha}|_{d_i} = \kappa d_{i}(\alpha(0), \alpha(1)) \le \kappa^{2}d(\alpha(0), \alpha(1)) \le  \kappa^2 |\alpha|_d.$$  

Moreover, since $\beta_{\alpha} (0) = \alpha(0)$ and $\beta_{\alpha}(1) = \alpha(1)$ lie in a convex set $P \cap \tilde{K}$, it follows that $\beta_{\alpha}$ lies in $P \cap \tilde{K}$. We note that $P \cap \tilde{K}$ is a subset of $\tilde{K}_{S_i}$, it follows that $\beta_{\alpha}$ lies in $\tilde{K}_{S_i}$. By replacing every maximum subpath $\alpha$ of $\gamma$ that does no lie in $\tilde{K}_{S_i}$ by the path $\beta_{\alpha}$ (as defined above), we obtain a new path in $\tilde{K}_{S_i}$ connecting $\tilde{s}_0$ to $h(\tilde{s}_0)$ whose length is no more than $\kappa^2\,|\gamma|_{d}$ with respect to the $d$--metric.
\end{proof}

\begin{rem}\label{rem:vertexandedgegroup}
The proof of Lemma \ref{lem:undistorted} also implies that any vertex subgroup and any subgroup of an edge group of $\pi_1(M)$ is undistorted.
\end{rem}

We give a proof for the lower bound part of subgroup distortion in Theorem~\ref{technical}.
\begin{proof}[Proof of Theorem~\ref{technical}, lower bound part]
Let $S_1, S_2, \cdots, S_n$ be the connected components of the almost fiber surface $\Phi(H)$. For each $i \in \{1, \cdots, n\}$, by Lemma~\ref{lem:undistorted}, $\pi_1(S_i)$ is undistorted in $H$. By applying Lemma~\ref{lem:preparationdist} to $G' = G$, $K = H$, and $K' = \pi_1(S_i)$, we have that $$\delta_{S_i} =\Delta_{\pi_1(S_i)}^{G} \preceq \Delta_{H}^{G}$$
Since $f(n):=\max\{\delta_{S_i}(n)\ |\ S_i\text{\ is\ a\ component\ of\ }\Phi(H)\}$. It follows that $f \preceq \Delta_{H}^{G}$.
\end{proof}

\bigskip

\subsection{Upper bound of subgroup distortion}\label{upperbound}
In this section, we prove the upper bound part of Theorem \ref{technical}, which is more complicated than the lower bound. The proof consists of two steps of reductions. 

{\bf First reduction:}
Now we do the first step for proving the upper bound part of Theorem \ref{technical}: roughly we throw away all $K_i\subset K$ corresponding to $S^1$-bundle or geometrically finite pieces of $M_H^c$.

Let $A$ be the complement in $K$ of the intersection of $K$ with $S^1$-bundle and geometrically finite pieces of $M_H^c$. Let $A_1, A_2,\cdots,A_l$ be the components of $A$, let $g$ be the maximum of $\{\Delta^G_{\pi_1(A_j)}\ |\ j=1,2,\cdots,l\}$, and let $\bar{g}$ be the superadditive closure of $g$. Then we have:

\begin{prop}\label{firstreduction}
  $$\Delta^G_H\preceq \bar{g}.$$
\end{prop}

\begin{rem}
  The subgroups $\pi_1(S_i)$ in Theorem \ref{technical} are just proper subgroups of $\pi_1(A_j)$. To construct $\Phi(H)$, we paste virtual fiber and partial fiber surfaces in pieces of $M^c_H$ (with nontrivial $\pi_1$) only along circles. On the other hand, $A$ is a union of Scott cores of pieces of $M^c_H$ that are virtual fiber pieces, partial fiber pieces, as well as simply-connected pieces that cover Seifert pieces of $M$. (Although simply-connected pieces are partial fiber pieces by definition, we sometimes single them out since they are quite special.) Moreover, we paste these Scott cores along both annuli and discs. Here we paste two Scott cores along a disc if and  only if they lie in adjacent partial fiber pieces of $M^c_H$.
\end{rem}

%Now we prove Proposition \ref{firstreduction}.
\begin{proof}[Proof of Proposition \ref{firstreduction}]
If $M_{H,i}$ is an $S^1$-bundle piece or geometrically finite piece of $M^{c}_H$, let $R_i$ be the constant given by Lemmas \ref{easypiece}. If $M_{H,i}$ is a partial fiber piece of $M^{c}_H$, let $R_i$ be the constant given by Lemma~\ref{hardpiece}. Since $M^{c}_H$ has only finitely many pieces, we let $R$ be the maximum of the numbers $R_i$ chosen above. We can assume that $R >1$. Let $\kappa>1$ be the constant given by the end of Preparation Step I. Let $\rho>0$ be the lower bound of the $d$-distance for any pair of distinct edge spaces in $\tilde{M}$.
 % We have a few constants: $\kappa>1$ by the end of Preparation Step I, $R>1$ right after Lemma \ref{easypiece}, and $\rho>0$ provides a lower bound of the $d$-distance for any pair of distinct edge spaces in $\tilde{M}$.

 %Let $R$ be the maximum of all $R_i$ in Lemmas \ref{easypiece} and \ref{hardpiece}. We can assume that $R >1$. Let $\kappa>1$ by the end of Preparation Step I, and  $\rho>0$ be the lower bound of the $d$-distance for any pair of distinct edge spaces in $\tilde{M}$.
 % We have a few constants: $\kappa>1$ by the end of Preparation Step I, $R>1$ right after Lemma \ref{easypiece}, and $\rho>0$ provides a lower bound of the $d$-distance for any pair of distinct edge spaces in $\tilde{M}$.
We are going to prove that
 $$\max{\{d_{\tilde{K}}(x,y)\ |\ x,y\in \tilde{K} \text{\ and\ } d_{\tilde{M}}(x,y)\leq n\}} \preceq \bar{g}(n)$$
 
 We briefly describe here the idea of the proof. For any $x,y\in \tilde{K}$ with $d_{\tilde{M}}(x,y)\leq n$, we are going to construct a path $\gamma'''$ in $\tilde{M}$ from $x$ to $y$, such that $\abs{\gamma'''}$ is bounded above by a linear function of $n$ (depends only on $\rho, R, \kappa$). Moreover, $\gamma'''$ can be written as a concatenation of finitely many subpaths, such that each subpath either lies in $\tilde{K}$ or homotopic to a path in the preimage of $\cup_{j=1}^l A_j$ (relative to its boundary). We then use the same argument as in the proof of Theorem~4.1 in \cite{Ngu18a}  to conclude that $$\max{\{d_{\tilde{K}}(x,y)\ |\ x,y\in \tilde{K} \text{\ and\ } d_{\tilde{M}}(x,y)\leq n\}} \preceq \bar{g}(n).$$

  Without loss of generality, since $K$ is compact, we assume both $x$ and $y$ lie in edge spaces of $\tilde{M}$. Let $\gamma'$ be a shortest geodesic in $(\tilde{M}, d$) connecting $x$ and $y$, then $|\gamma'| \le n$.
   As in \cite{HN19} and \cite{Ngu18a} (see the paragraph above Claim~1 in the proof of Theorem~6.1 in \cite{HN19} and the eighth paragraph in the proof of Theorem~4.1 in \cite{Ngu18a} respectively), we can replace $\gamma'$ by $\gamma$, such that $\gamma$ is transverse to edge spaces of $\tilde{M}$, it intersects with each edge space of $\tilde{M}$ at most once, and $|\gamma| \le \kappa^2 |\gamma'| \le \kappa^2 n$.

  Let $y_0=x, y_k=y$, and let $y_1,y_2,\cdots,y_{k-1}$ be the other intersection points between $\gamma$ and edge spaces of $\tilde{M}$, and we denote these edge spaces by $E_1,E_2,\cdots E_{k-1}$, with $y_i\in E_i$. Let $\gamma_i$ be the subpath of $\gamma$ from $y_{i-1}$ to $y_i$, and let $\tilde{M_i}$ be the piece of $\tilde{M}$ that contains $\gamma_i$. By the choice of $\rho$, we have $\rho k \le |\gamma| \le \kappa^2 n$, and $k \le \frac{\kappa^2}{\rho}n$.

  For each $\gamma_i$, if it lies in a piece of $\tilde{M}$ that covers a finite cover piece of $M_H^c$, then it is completely contained in $\tilde{K}$. If $\gamma_i$ lies in a piece $\tilde{M}_i\subset \tilde{M}$ that covers a $S^1$-bundle or a geometrically finite piece of $M_H^c$, we modify $\gamma_i$ as the following.

  We apply Lemma \ref{easypiece} to the two endpoints of $\gamma_i$ ($y_{i-1}$ and $y_i$) and get the following two paths (we refer the reader to Figure~\ref{fig1} for a schematic picture): a path $\alpha_i$ connecting $y_{i-1}$ and $y_i$, and a path $\beta_i$ in $\tilde{K}_i$ connecting points $z_{i-1}' \in E_{i-1}$ and $z_i\in E_i$, such that the following holds:
  \begin{enumerate}
    \item 
    \label{item 1}$d_i(z_{i-1}',\alpha_i),d_i(z_i,\alpha_i) \le R$,
    \item
    \label{item 2}$|\alpha_i|_{d_i} = d_{i}(\gamma_{i}(0), \gamma_{i}(1))$,
    \item 
    \label{item 3}
    $|\beta_i|_{d_i} \le Rd_i(z_{i-1}',z_i)$.
  \end{enumerate}

Then we get a concatenation of paths $\xi_{i-1}\beta_i\xi_i'$ that is homotopic to $\gamma_i$, where $\xi_{i-1}$ is the geodesic path in $E_{i-1}$ from $y_{i-1}$ to $z_{i-1}'$, and $\xi_i'$ is the geodesic path in $E_i$ from $z_i$ to $y_i$. By the choice of $d_i$-metric, for any two points in the same boundary component of $M_i$, the shortest path in $M_i$ (under the $d_i$-metric) between these two points still lie in this boundary component.

Let $\{\gamma_{i_0}, \gamma_{i_1}, \cdots, \gamma_{i_s}\}$ be the collection of all subpaths $\gamma_i$ of $\gamma$ where each $\gamma_i$ lies in a piece that covers a $S^1$--bundle or a geometrically finite piece of $M_{H}^{c}$. Let $\gamma''$ be the path obtained by replacing each $\gamma_{i_j}$ (with $j =1,\cdots, s$) by $\xi_{i_j-1}\beta_{i_j}\xi_{i_j}'$.

\begin{center}
\includegraphics[width=2.5in]{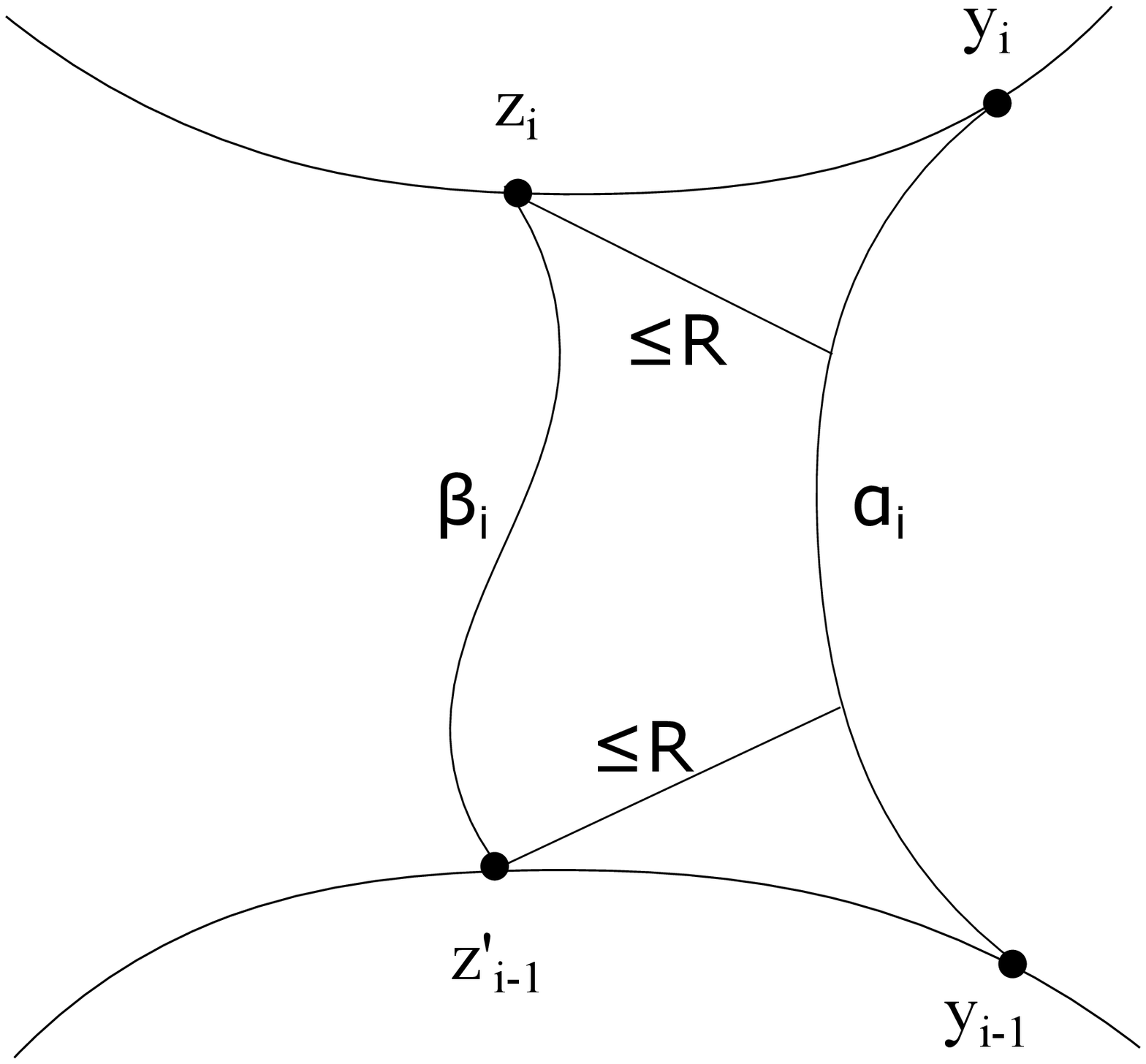}
\begin{figure}[!ht]
\caption{The path $\alpha_{i}$ connects the initial point of $\gamma_i$ on the plane $E_{i-1}$ to the terminal point of $\gamma_i$ on the plane $E_i$. The endpoints $z'_{i-1}$ and $z_i$ of the path $\beta_i \subset \tilde{K}$ are within $R$--neighborhood of $\alpha_i$}
\label{fig1}
\end{figure}
 %\centerline{Figure 1}
\end{center}

%\begin{center}
%\includegraphics[width=4in]{figure2.eps}
% \centerline{Figure 2}
%\end{center}

{\bf{Claim:}}
There exists a linear function $J$ depending only on $\kappa, R$ and $\rho$ such that $\abs{\gamma''}_{d} \le J(n)$.

Indeed, let $j$ be an element in $\{i_0, i_1, \cdots, i_s\}$. Using (\ref{item 1}), (\ref{item 2}), and (\ref{item 3}) we have that $d_{j}(y_{j-1}, z_{j-1}') \le R + \abs{\alpha_j}_{d_j}$, $d_{j}(z_j, y_j) \le R + \abs{\alpha_j}_{d_j}$, and $d_{j}({z_{j-1}'}, z_j) \le 2R + \abs{\alpha_j}_{d_j}$. We recall that the relation between the metrics $d$ and $d_j$ is discussed in Preparation Step I. We have
\begin{align*}
    \bigl |\xi_{j-1}\beta_j \xi_j' \bigr | &\leq \kappa \bigl (|\xi_{j-1}|_{d_j}+|\beta_j|_{d_j}+|\xi_j'|_{d_j} \bigr ) \\
    &\leq \kappa \bigl (d_j(y_{j-1},z_{j-1}')+d_j(z_j,y_j)+Rd_j(z_{j-1}',z_j) \bigr )\\
    &\le \kappa \bigl (  R + \abs{\alpha_j}_{d_j} + R + \abs{\alpha_j}_{d_j} + R(2R + \abs{\alpha_j}_{d_j}) \bigr )\\
    &= \kappa \bigl ( 2R + 2 \abs{\alpha_j}_{d_j} + R(2R + \abs{\alpha_j}_{d_j}) \bigr )\\
    &\le \kappa \bigl ( 2R^{2} + 2R\abs{\alpha_j}_{d_j} + R(2R + \abs{\alpha_j}_{d_j}) \bigr ) \,\,\,\textup{using $R \ge 1$}\\
    &= \kappa R (4R + 3 \abs{\alpha_j}_{d_j} ) \le \kappa R (4R + 3R \abs{\gamma_j}_{d_j} ) = \kappa R^2  (4 + 3 \abs{\gamma_j}_{d_j})\\
    &\le \kappa R^{2} (4 + 3\kappa \abs{\gamma_j}) < \kappa ^{2} R^{2} ( 4 + 3 \abs{\gamma_j})
\end{align*}
Summing over $j$, we have
\begin{align*}
\sum_{j=i_0}^{i_s}   \bigl |\xi_{j-1}\beta_j \xi_j' \bigr | &\le \sum_{j=i_0}^{i_s} \kappa^{2} R^2 \bigl (4 + 3 \abs{\gamma_j} \bigr ) =  \sum_{j=i_0}^{i_s} 4 \kappa^2 R^2 +  \sum_{j=i_0}^{i_s} 3\kappa^2 R^2 \abs{\gamma_j}\\
&\le 4 \kappa^2 R^2 k +  3 \kappa^2 R^2 \abs{\gamma} \le  4 \kappa^2 R^2 k +  3 \kappa^4 R^2 \abs{\gamma'}  \le 4 \kappa^4 R^2 n/\rho +  3 \kappa^4 R^2n
\end{align*}
Since the sum of lengths of the subpaths in the complement of $\xi_{j-1} \beta _{j} \xi_j'$ (in $\gamma''$) is no more than $\abs \gamma$ that is less than $\kappa^2 n$, it follows that 
\[
\abs{\gamma''} \le \kappa^2 n + 4\kappa^{4} R^2 n/\rho + 3\kappa^4 R^2 n = (\kappa^2 + 4\kappa^4 R^2/\rho + 3\kappa^4 R^2)n
\]
Let $J(x) = (\kappa^2 + 4\kappa^{4} R^2/\rho + 3\kappa^4 R^2)x$, the claim is confirmed.

%By the choice of $d_i$ metric, for any two points in the same boundary component of $M_i$, the shortest path in $M_i$ (under the $d_i$-metric) between these two points still lie in the boundary. So we have  $$|\xi_{i-1}\beta_i\xi_i'|\leq \kappa(|\xi_{i-1}|_{d_i}+|\beta_i|_{d_i}+|\xi_i'|_{d_i}) \leq \kappa (d_i(y_{i-1},z_{i-1})+d_i(z_i,y_i)+Rd_i(z_{i-1},z_i))$$
 % $$\leq \kappa R(|\alpha_i|_{d_i}+4R)\leq \kappa R^2(|\gamma_i|_{d_i}+4)\leq \kappa^2 R^2(|\gamma_i|+4).$$

 %We replace each such $\gamma_i$ by $\xi_{i-1}\beta_i\xi_i'$, to get a new path $\gamma''$ from $x$ to $y$ whose $d$-length is bounded above by $$\sum_{i=1}^k \kappa^2 R^2(|\gamma_i|+4)\leq \kappa^2 R^2 |\gamma|+4\kappa^2 R^2 k \leq \kappa^4 R^2|\gamma'|+4\kappa^2 R^2\frac{\kappa^2}{\rho}n\leq \kappa^4R^2 (1+\frac{4}{\rho})n.$$
 
 We recall that each $\beta_{i_j}$ (with $j \in \{1, \cdots, s\}$) is contained in $\tilde{M}_{i_j}$, and $\beta_{i_j} \subset \tilde{K}$ also holds. Whenever $\tilde{M}_{i_j}$ and $\tilde{M}_{i_t}$ are adjacent with $i_t = i_j +1$, and we did the above path modification for both $\gamma_{i_j}$ and $\gamma_{i_t}$, we replace the subpath $\xi_{i_j}' \cdot \xi_{i_j}$ of $\gamma''$ by the geodesic $\zeta _{i_j} = [z_{i_j}, z_{i_j}']$ in the plane $E_{i_j}$ connecting $z_{i_j}$ to $z_{i_j}'$. We note that $z_{i_j}$ and $z_{i_j}'$ lie in $\tilde{K}$ since they are endpoints of $\beta_{i_j} \subset \tilde{K}$ and $\beta_{i_t} \subset \tilde{K}$ respectively. By convexity of $\tilde{K} \cap E_{i_j}$, the geodesic $\zeta_{i_j}$ in $E_{i_j}$ must lie in $\tilde{K}$. Moreover, by the triangle inequality we have 
 \[|\zeta_{i_j}| = d(z_{i_j}, z_{i_j}') \le d(z_{i_j}, y_{i_j}) + d(y_{i_j}, z_{i_j}') =\abs{\xi_{i_j}'} + \abs{\xi_{i_j}},
 \]
 so the length of the new path $\gamma'''$ is no more than the length of $\gamma''$ that is bounded above by the linear function $J(n)$.
 
 Thus, we obtain a new path $\gamma'''$ who can be written as a concatenation of finitely many subpaths, and each subpath lies in one of the following to cases:
 \begin{itemize}
     \item It is either a subpath that lies in $\tilde{K}$, including: $\beta_{i_j}$'s, $\zeta_{i_j}$'s and original $\gamma_j$'s that lie in finite cover pieces.
     \item or it is homotopic to a path lying in one elevation of some $A_j$ relative to boundary, and such a subpath can be written as $\xi_{j-1}'\cdot \gamma_j\cdot \xi_j$. Here $\gamma_j$ is a component of the complement of the subpaths $\{\gamma_{i_0},\cdots,\gamma_{i_s}\}$ (that lie in geometrically finite or $S^1$-bundle pieces), while $\xi_{j-1}'$ and $\xi_j$ are obtained by applying the above construction to $\gamma_{j-1}$ and $\gamma_j$. 
 \end{itemize} In the second possibility, by the construction of $\xi_j$'s, the initial and terminal points of $\xi_{j-1}'\cdot \gamma_j\cdot \xi_j$ lie in $\tilde{K}$, and they actually lie in the same component $A_j$ of the elevation of $A$. So it is homotopy to a path $\gamma_j'$ in $A_j$, with the end points fixed. 
 
 % So $|\gamma_j'| \preceq \Delta^{\pi_1(M)}_{\pi_1(A_j)}(|\gamma_j|)$
 
Let $\Delta_j$ be the distortion of the chosen elevation of $A_j$ in $\tilde{M}$, then $\Delta_{j} \sim \Delta^{\pi_1(M)}_{\pi_1(A_j)}$ (in the sense of Definition~\ref{def:equivalentfunction}). Let $\gamma'_{j}$ be the shortest path in the elevation of $A_j$ connecting the two endpoints of $\xi_{j-1}'\cdot \gamma_j\cdot \xi_j$. It follows from the definition of $\Delta_j$ that $\Delta_{j}(|\xi_{j-1}'\cdot \gamma_j\cdot \xi_j|) \ge |\gamma'_j|$. By the same argument as in the last paragraph of the proof of Theorem~4.1 in \cite{Ngu18a}, the proof is done.
 %   This new path $\gamma''$ is a concatenation $\zeta_0'\beta_1\zeta_1'\beta_2\cdots \zeta_{k-1}'\beta_k \zeta_k'$, with $\zeta_i'=\xi_i'\xi_i$ contained in $E_i$. Here each $\beta_i$ is contained in $\tilde{M}_i$, and $\beta_i\subset \tilde{K}$ if $\tilde{M}_i$ covers a finite cover, $S^1$-bundle or geometrically finite piece of $M_H$. If for some $\zeta_i'$, both its initial and terminal points lie in $\tilde{K}\cap E_i$, by convexity of $\tilde{K}\cap E_i$, the shortest geodesic $\zeta_i$ in $E_i$ connecting these two points lie in $\tilde{K}$. By replacing each $\zeta_i'$ by $\zeta_i$, we get a new  path whose length is still bounded by $\kappa^4R^2 (1+\frac{4}{\rho})n$.
%  The further modified path $\gamma'''=\zeta_0\beta_1\zeta_1\beta_2\cdots \zeta_{k-1}\beta_k \zeta_k$ is the path from $x$ to $y$ we are aiming for. I f one segment ($\beta_i$ or $\zeta_i$) does not lie in $\tilde{K}$, then it lies in a piece of $\tilde{M}$ that covers a fiber or partial fiber piece of $M_H$ (including the possibility that this piece of $\tilde{M}$ is simply connected and it covers a Seifert piece of $M$). See Figure 2 for a cartoon of the modification process from $\gamma'$ to $\gamma'''$.
%By fusing some segments in $\gamma'''$, we can rewrite $\gamma'''$ as a concatenation of finitely many subpaths, such that each subpath either lies in $\tilde{K}$, or homotopy to a path in the preimage of $\cup_{j=1}^l A_j$ (relative to its boundary). Since $|\gamma'''|<Cn$ for $C=\kappa^4R^2(1+\frac{4}{\rho})$, by the same argument as in \cite{Ngu}, the proof is done.
\end{proof}

{\bf Second reduction:} The subset $A$ obtained in the step I is a union of Scott cores of virtual fiber pieces and partial fiber pieces (including those pieces with trivial or infinitely cyclic fundamental groups) of $M^c_H$, pasting along annuli and discs. Moreover, $A$ is homeomorphic to a surface cross the interval, and this surface is obtained by pasting its vertex pieces along circles and arcs. The almost fiber surface $\Phi(H)$ is naturally a subsurface of the above surface, obtained by pasting those pieces with non-cyclic fundamental groups along circles. In this step, we prove that the subgroup distortion of components of $A$ are determined by the corresponding components of $\Phi(H)$.

Now we enlarge $\Phi(H)$ to get a new surface $\Phi(H)'$ by the following way: for each partial fiber piece of $M^c_H$ with infinite cyclic fundamental group, add an annulus to $\Phi(H)$ and paste it to the original $\Phi(H)$ along a circle if possible; for each partial fiber piece of $M^c_H$ with trivial fundamental group, add a disc to $\Phi(H)$ as a new component of $\Phi(H)$. The advantage of this new surface $\Phi(H)'$ is that it intersects with each piece of $A$ nontrivially. 

By construction, each component of $\Phi(H)'$ either deformation retracts to a component of $\Phi(H)$, or is a disc, or is an annulus. In the third case, the annulus subgroup is horizontal in a Seifert piece of $\pi_1(M)$, so it is undistorted in the vertex subgroup. By Remark \ref{rem:vertexandedgegroup}, any vertex subgroup is undistorted in $\pi_1(M)$, so the annulus subgroup
is undistorted in $\pi_1(M)$. In conclusion, the $\bar{f}$ defined for $\Phi(H)'$ is equivalent to the $\bar{f}$ defined for $\Phi(H)$. So we only need to compare the distortions of $\pi_1(A)$ and $\pi_1(\Phi(H)')$, and we will abuse notation to still denote $\Phi(H)'$ by $\Phi(H)$.

For simplicity, we assume that $A$ has only one component. Let $S_1,\cdots,S_m$ be the components of $\Phi(H)$ with $\Phi(H)=\cup_{i=1}^m S_i$. For each $i$, let $B_i$ be the union of pieces of $A$ that intersects with $S_i$ nontrivially, then each $B_i$ is homeomorphic to $S_i\times I$, and $A$ can be obtained by pasting $\{B_i\}_{i=1}^m$ along discs. Since the inclusion $S_i\subset B_i$ is a homotopy equivalence, we only need to bound the distortion of $\pi_1(A)$ by the distortions of $\{\pi_1(B_i)\}_{i=1}^m$.

Let $h$ be the maximum of subgroup distortions $\{\Delta^G_{\pi_1(B_i)}\}_{i=1}^m$, and let $\bar{h}$ be its superadditive closure, then we prove the following result.

\begin{prop}\label{secondreduction}
$$\Delta^G_{\pi_1(A)}\preceq \bar{h}.$$
\end{prop}

\begin{proof}
    %We have a few constants: $\kappa>1$ by the end of Preparation Step I, $D>0$ by the end of Preparation Step II, $R>1$ right after Lemma \ref{hardpiece}, and $\rho>0$ at the beginning of the proof of Proposition \ref{firstreduction}.
Let $D >0$ be the constant given by the end of Preparation Step II. Let $R>1$, $\kappa>1$ and $\rho>0$ be the constants given by the first paragraph of the proof of Theorem~\ref{firstreduction}.

  Let $\tilde{A}$ be one elevation of $A$ in $\tilde{M}$, then we need to prove that $$\max{\{d_{\tilde{A}}(x,y)\ |\ x,y\in \tilde{A} \text{\ and\ } d_{\tilde{M}}(x,y)\leq n\}} \preceq \bar{h}(n).$$

  We briefly describe here the idea of the proof. For any $x,y\in \tilde{A}$ with $d_{\tilde{M}}(x,y)\leq n$, we will construct a path $\gamma'''$ in $\tilde{M}$ connecting $x$ to $y$ such that the following holds:  the path $\gamma'''$ can be written as a concatenation of subpaths, such that each subpath connects two points in an elevation of $B_i$ for some $i$. Moreover, there exists  a linear function $F(n)$ (only depends on $D$, $R$, $\kappa$ and $\rho$) such that $\abs{\gamma'''}$ is bounded above by $F(n)$.

  Since $A$ is compact, we can assume both $x$ and $y$ lie in edge spaces of $\tilde{M}$. Let $\gamma'$ be the shortest path in $(\tilde{M},d)$ connecting $x$ and $y$, then $|\gamma'| \le n$. As in the proof of Proposition \ref{firstreduction}, we can replace $\gamma'$ by $\gamma$ such that it is transverse to edge spaces of $\tilde{M}$, it intersects with each edge space of $\tilde{M}$ only once, and $|\gamma| \le \kappa^2 |\gamma'| \le \kappa^2 n$.

  Instead of taking the intersection of $\gamma$ with all edge spaces in $\tilde{M}$, we take the intersection of $\gamma$ with all edge spaces in $\tilde{M}$ that are mapped to plane edge spaces of $M_H$. By adding $x$ and $y$ to these points, we get $y_0=x,y_1,\cdots,y_{k-1},y_k=y$ such that $y_i$ lies in an edge space $E_i\subset \tilde{M}$. As in the proof of Proposition \ref{firstreduction}, $k\leq \frac{\kappa^2}{\rho}n$ holds. Let $\gamma_i$ be the subpath of $\gamma$ from $y_{i-1}$ to $y_i$, then $\gamma$ is a concatenation of $\gamma_i$: $\gamma=\gamma_1 \cdots \gamma_k$.

  These $\gamma_i$'s are important in this proof, since each $\gamma_i$ is homotopic to a path in an elevation of some $B_j$, such that the homotopy process keep the two endpoints of $\gamma_i$ lying in $E_{i-1}$ and $E_i$ respectively. Moreover, each $\gamma_i$ is a concatenation of a few (possibly one) subpaths such that each such subpath lies in a piece of $\tilde{M}$ properly. We call the first subpath the initial path of $\gamma_i$ and call the last one the terminal path of $\gamma_i$ (they might be same with each other).

  For any $i\in \{1,\cdots,k-1\}$, since $A$ is constructed by pasting Scott cores of virtual fiber and partial fiber pieces of $M_H$ together (including simply connected pieces), for the two pieces of $\tilde{M}$ adjacent to $E_i$, both of them cover partial fiber pieces in $M_H$ and Seifert pieces in $M$. So there are two lifted fibering structures on the plane $E_i$, one is from the piece of $\tilde{M}$ that contains the terminal path of $\gamma_{i}$ (called left fibration), and the other one is from the piece of $\tilde{M}$ that contains the initial path of $\gamma_{i+1}$ (called right fibration). These two fibering structures on $E_i$ both consist of geodesic lines as leaves and they are distinct from each other.

  In the following, we modify the $\gamma_i$ by two steps, such that its initial and terminal points lie in $\tilde{A}$. We recall that it is possible that two planes $E_{i-1}$ and $E_{i}$ may not lie in the same piece of $\tilde{M}$.

  %For each $i \in \{1, \cdots, k-1\}$, let $\beta_{j_i}$ and $\beta_{j_i + 1}$ be the paths given by Lemma~\ref{hardpiece} corresponding to the terminal path  of $\gamma_{i-1}$ and initial path $\gamma_i$ respectively. For each $\beta_s$ given above, we denote the initial point and the terminal point of each $\beta_{s}$ by $z_{s-1}'$ and $z_{s}$ respectively. Let $\zeta_{s}$ be the geodesic in the plane that contains both $z_{s}$ and $z_{s}'$ connecting $z_s$ to $z_{s}'$.

  {\bf Step I:}
  The first step is similar to the modification process in the proof of Proposition \ref{firstreduction}. Instead of applying Lemma \ref{easypiece} to (some) $\gamma_i$ as in Proposition \ref{firstreduction}, we apply Lemma \ref{hardpiece} to the initial and terminal paths of $\gamma_i$, except the initial path of $\gamma_1$ and the terminal path of $\gamma_k$. 
  
  More precisely, for each $i\in \{1, \cdots, k\}$, let $\gamma_i^{\text{ini}}$ and $\gamma_i^{\text{ter}}$ be the initial and terminal paths of $\gamma_i$ respectively (it is possible that $\gamma_i^{\text{ini}}=\gamma_i^{\text{ter}}$). We apply Lemma~\ref{hardpiece} to $\gamma_i^{\text{ini}}$ and $\gamma_i^{\text{ter}}$ to get $\beta_i^{\text{ini}}$ and $\beta_i^{\text{ter}}$ respectively (except $\gamma_1^{\text{ini}}$ and $\gamma_k^{\text{ter}}$). Note that $\gamma_i^{\text{ini}}$ and $\beta_i^{\text{ini}}$ have different end points, and the same for $\gamma_i^{\text{ter}}$ and $\beta_i^{\text{ter}}$. By concatenating with geodesic paths in edge spaces of $\tilde{M}$, we get $\delta_i^{\text{ini}}\beta_i^{\text{ini}}\eta_i^{\text{ini}}$ in the same piece of $\tilde{M}$ as $\gamma_i^{\text{ini}}$ and they have the same endpoints. Similarly, we get $\delta_i^{\text{ter}}\beta_i^{\text{ter}}\eta_i^{\text{ter}}$ for $\gamma_i^{\text{ter}}$. Moreover, we delete the initial and terminal paths of $\gamma_i$ to get $\hat{\gamma}_i$ (which might be empty). 
  
  %In this way, we replace the old path $\gamma=\gamma_1\cdots\gamma_k$ by a new path $\gamma''$ 
  %\[
  %\cdots (\beta_{j_1}\zeta_{j_1}\beta_{j_1 +1}) \cdots (\beta_{j_2}\zeta_{j_2}\beta_{j_2 +1}) \cdots (\beta_{j_{k-1}}\zeta_{j_{k-1}}\beta_{j_{k-1} +1}) \cdots\] 
  %Here the symbol $\cdots$ between $(\beta_{j_i}\zeta_{j_i}\beta_{j_i +1})$ and $(\beta_{j_{i+1}}\zeta_{j_{i+1}}\beta_{j_{i+1} +1})$ is the subpath of $\gamma$ connecting the terminal point of $ \beta_{j_i +1}$ to the initial point of  $ \beta_{j_{i+1}}$. 
  
  In this way, we replace the old path $\gamma=\gamma_1\cdots\gamma_k$ by a new path
  \[
  (\gamma_1^{\text{ini}})\hat{\gamma_1}(\delta_1^{\text{ter}}\beta_1^{\text{ter}}\eta_1^{\text{ter}})(\delta_2^{\text{ini}}\beta_2^{\text{ini}}\eta_2^{\text{ini}})\hat{\gamma_2}(\delta_2^{\text{ter}}\beta_2^{\text{ter}}\eta_2^{\text{ter}})\cdots(\delta_k^{\text{ini}}\beta_k^{\text{ini}}\eta_k^{\text{ini}})\hat{\gamma_k}(\gamma_k^{\text{ter}})
  \]
  Note that it is possible that $\hat{\gamma}_i$ is the empty path. In this case, it is possible that $\delta_i^{\text{ini}}\beta_i^{\text{ini}}\eta_i^{\text{ini}}$ and $\delta_i^{\text{ter}}\beta_i^{\text{ter}}\eta_i^{\text{ter}}$ are the same path, and we only have one of them in the above concatenation.
  
  Since both $\eta_{i-1}^{\text{ter}}$ and $\delta_i^{\text{ini}}$ lie in the same edge space $E_i\subset \tilde{M}$ (that contains $y_i$), we can replace the concatenation $\eta_{i-1}^{\text{ter}} \delta_i^{\text{ini}}$ by the geodesic $\zeta_i$ in $E_i$ with the same endpoints: from $z_i$ to $z_i'$. Now we denote the concatenation $\beta_i^{\text{ini}}\eta_i^{\text{ini}}\hat{\gamma_i}\delta_i^{\text{ter}}\beta_i^{\text{ter}}$ by $\xi_i$ (its definition is slightly different for $\xi_1$ and $\xi_k$), and get our new path $\gamma''$:
  
  \[\gamma''=\xi_1\zeta_1\xi_2\zeta_2\cdots\xi_{k-1}\zeta_{k}\xi_k\]
  
  The new path $\gamma''$ has the following properties: 
   %$\gamma''=\zeta_0\beta_1\zeta_1\cdots \zeta_{k-1}\beta_k\zeta_k$ such that the following holds:
  \begin{enumerate}
    \item $\gamma$ and $\gamma''$ have the same initial and terminal points in $\tilde{A}$.
    %\item $|\gamma''| <\kappa^4R^2(1+\frac{4}{\rho}) n$.
   \item $\abs{\gamma''} \le J(n)$ where $J(n)=(\kappa^2 + 4\kappa^{4} R^2/\rho + 3\kappa^4 R^2)n$ is the linear function given by the proof of Proposition \ref{firstreduction}. (Apply Lemma \ref{hardpiece} instead of Lemma \ref{easypiece}.)
    \item Each $\zeta_i$ is contained in the edge space $E_i\subset \tilde{M}$ that contains $y_i$, and it is a geodesic in $E_i$.
   % \item Let $z_{i-1}'$ and $z_i$ be the initial and terminal points of $\beta_i$, then both $z_{i-1}'$ and $y_{i-1}$ lie in $E_{i-1}$, and both $z_i$ and $y_i$ lie in $E_i$.
    \item 
    \label{item:intersection}Let $\overleftarrow{\ell_i}$ be the leaf of the left fibration in $E_i$ going through $z_i$, and let $\overrightarrow{\ell_i}$ be the leaf of the right fibration in $E_i$ going through $z_i'$, then both $\overleftarrow{\ell_i}$ and $\overrightarrow{\ell_i}$ intersect with $\tilde{A}$.
  \end{enumerate}
Here condition (\ref{item:intersection}) uses Lemma \ref{hardpiece} condition (1), which is important in the following Step II.

 We note that if each $\zeta_i$ intersects with $\tilde{A}$ nontrivially, then our proof is done. We can rewrite $\gamma''$ as another concatenation of paths, by splitting $\zeta_i$ into two subpaths and fusing them with $\xi_i$ and $\xi_{i+1}$. Then for each subpath in this new concatenation, it connects two points in the same elevation of $B_i$ for some $i$. Then we can take $\gamma'''$ to be this new concatenation.

 {\bf Step II:}
  In general, $\zeta_i$ may not intersect with $\tilde{A}$ as we wish, and we will use Euclidean geometry to show that its distance from $\tilde{A}$ in $E_i$ is bounded.

   For each decomposition torus $T\subset M$ that is adjacent to two Seifert pieces, it has two distinct induced fibrations consist of geodesic leaves, and let $c_T$ be the angle between these two fibrations, with respect to the $d$-metric. Since $M$ has finitely many decomposition tori, there exists $c>0$ such that $c<c_T<\pi-c$ for all $T$.

  Let $d_{E_i}$ be the restriction of $d$ on $E_i$, then we prove that: $$d_{E_i}(\zeta_i,\tilde{A}\cap E_i)\le \csc{c}\cdot(|\zeta_i|+D).$$

  Let $o_i$ be the intersection point of $\overleftarrow{\ell_i}$ and $\overrightarrow{\ell_i}$ in $E_i$ (it exists since $\overleftarrow{\ell_i}$ and $\overrightarrow{\ell_i}$ are not parallel), let $x_i$ be one intersection point of $\overleftarrow{\ell_i}$ and $\tilde{A}\cap E_i$, and let $x_i'$ be one intersection point of $\overrightarrow{\ell_i}$ and $\tilde{A}\cap E_i$. Since the diameter of $\tilde{A}\cap E_i$ in $E_i$ is bounded by $D$, $d_{E_i}(x_i,x_i')<D$ holds. This Euclidean picture is shown in Figure 2.

\begin{center}
\includegraphics[width=2in]{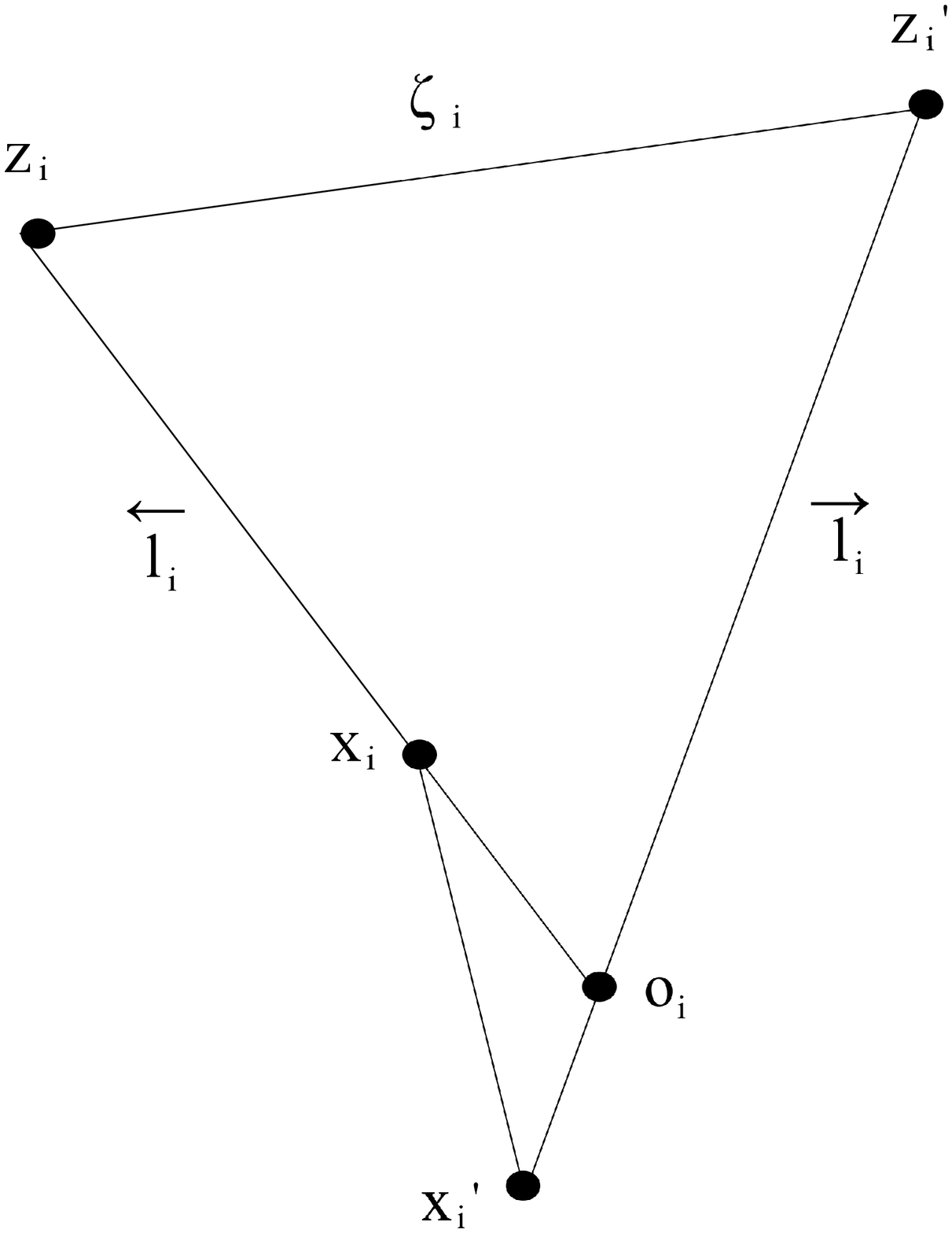}
 \centerline{Figure 2}
\end{center}

  Since $E_i$ has an induced flat metric, we can apply Euclidean geometry on it. By the sine law, we have 
  \begin{align*}
      d_{E_i}(\zeta_i,\tilde{A}\cap E_i) &\leq d_{E_i}(z_i,x_i)\leq d_{E_i}(z_i,o_i)+d_{E_i}(x_i,o_i)\\
      &=d_{E_i}(z_i,z_i')\frac{\sin{\angle z_iz_i'o_i}}{\sin{\angle z_io_iz_i'}}+d_{E_i}(x_i,x_i')\frac{\sin{\angle x_ix_i'o_i}}{\sin{\angle x_io_ix_i'}}\\
      &\le |\zeta_i|\frac{1}{\sin{c}}+D\frac{1}{\sin{c}}\\
      &=\csc{c}\cdot(|\zeta_i|+D)
  \end{align*}
  %$$d_{E_i}(\zeta_i,\tilde{A}\cap E_i)\leq %d_{E_i}(z_i,x_i)\leq d_{E_i}(z_i,o_i)+d_{E_i}(x_i,o_i)$$
  %$$=d_{E_i}(z_i,z_i')\frac{\sin{\angle z_iz_i'o_i}}{\sin{\angle z_io_iz_i'}}+d_{E_i}(x_i,x_i')\frac{\sin{\angle x_ix_i'o_i}}{\sin{\angle x_io_ix_i'}}\leq |\zeta_i|\frac{1}{\sin{c}}+D\frac{1}{\sin{c}}=\csc{c}\cdot(|\zeta_i|+D).$$

  So we can modify the above $\gamma''$ by adding the shortest geodesic in $E_i$ from $\zeta_i$ to $\tilde{A}\cap E_i$ and its inverse, to get a new path $\gamma'''$. 

  Since $\gamma''$ contains all $\zeta_i$ as subpaths, we have $\sum_{i=1}^{k-1}|\zeta_i|\leq |\gamma''|$. So the length of $\gamma'''$ is bounded by a linear function on $n$, by the following estimate:
  \begin{align*}
      |\gamma'''|&\leq |\gamma''|+2\sum_{i=1}^{k-1}\csc{c}\cdot(|\zeta_i|+D)\\
      &\leq |\gamma''|+2\csc{c}\cdot \sum_{i=1}^{k-1}|\zeta_i|+2k\csc{c}\cdot D\\
      &\leq (1+2\csc{c})|\gamma''|+2k\csc{c}\cdot D\\
      &\leq \bigl [(1+2\csc{c})(\kappa^2 + 4\kappa^{4} R^2/\rho + 3\kappa^4 R^2)+2\csc{c}\cdot D\frac{\kappa^2}{\rho} \bigr ]n
  \end{align*}
  
  %$$|\gamma'''|\leq |\gamma''|+2\sum_{i=1}^{k-1}\csc{c}\cdot(|\zeta_i|+D)\leq |\gamma''|+2\csc{c}\cdot \sum_{i=1}^{k-1}|\zeta_i|+2k\csc{c}\cdot D$$
  %$$\leq (1+2\csc{c})|\gamma''|+2k\csc{c}\cdot D\leq [(1+2\csc{c})\kappa^4R^2(1+\frac{4}{\rho})+2\csc{c}\cdot D\frac{\kappa^2}{\rho}]n.$$

Moreover, $\gamma'''$ can be rewritten as a concatenation of subpaths, such that each subpath connects two points in an elevation of $B_i$ for some $i$.

  Then the proof is done by following the same argument as in \cite{Ngu18a}.

\end{proof}

Now we are ready to prove the upper bound part of Theorem \ref{technical}.

\begin{proof}[Proof of Theorem~\ref{technical}, upper bound part]
  We delete those finite cover, $S^1$-bundle and geometrically finite pieces of $K$, denote their complement by $A$, and denote the components of $A$ by $A_1,\cdots, A_l$. For each $A_i$, let $S_{i1},\cdots,S_{im_i}$ be the components of $\Phi(H)$ that are contained in $A_i$. Then we have   $$\bigcup_{i=1}^l(\bigcup_{j=1}^{m_i}S_{ij})=\Phi(H).$$

  By Proposition \ref{firstreduction}, for any $n\in \mathbb{Z}_+$, there exists $i_1,\cdots,i_k\in \{1,\cdots,l\}$ and  $n_1,\cdots,n_k\in \mathbb{Z}_+$ such that $n_1+\cdots+n_k=n$ and  $$\Delta^G_H(n)\preceq \sum_{s=1}^k\Delta^G_{\pi_1(A_{i_s})}(n_s).$$

  Then by Proposition \ref{secondreduction}, (after replacing $\Phi(H)$ by $\Phi(H)'$, which does not affect the distortions) for each $s\in \{1,\cdots,k\}$, there exists $j_{s1},\cdots,j_{sl_s}\in \{1,\cdots,m_{i_s}\}$ and $m_{s1},\cdots,m_{sl_s}\in \mathbb{Z}_+$ such that $m_{s1}+\cdots+m_{sl_s}=n_s$ and $$\Delta^G_{\pi_1(A_{i_s})}(n_s)\preceq \sum_{t=1}^{l_s}\Delta^G_{\pi_1(S_{i_sj_{st}})}(m_{st}).$$

  So we have
  $$\Delta^G_H(n)\preceq \sum_{s=1}^k\Delta^G_{\pi_1(A_{i_s})}(n_s)\preceq \sum_{s=1}^k(\sum_{t=1}^{l_s}\Delta^G_{\pi_1(S_{i_sj_{st}})}(m_{st})).$$
  Since $\sum_{s=1}^k(\sum_{t=1}^{l_s}m_{st})=\sum_{s=1}^kn_s=n$, we get $$\Delta^G_H(n)\preceq \overline{\max\{\Delta^G_{\pi_1(S_i)}\ |\ S_i\text{\ is\ a\ component\ of\ }\Phi(H)\}} (n),$$ where the overline means the superadditive closure. So the proof is done.
\end{proof}

\bigskip
\bigskip

\section{Distortion of finitely generated subgroups in 3-manifold groups}
\label{sec:distortiongeometricmld}
In this section, we are going to prove Theorem \ref{toriboundary} and Theorem \ref{general}.

\begin{proof}[Proof of  Theorem \ref{toriboundary}]
At first, we pass to a finite cover of $M$, such that each Seifert piece $M_i$ of $M$ is homeomorphic to $F_i\times S^1$, and $M$ does not contain the twisted $I$-bundle over Klein bottle.

Let $f$ be the function given by Theorem~\ref{technical}.
If $\Phi(H)$ is empty then we use the convention that $f(n) =0$. Since the zero function is equivalent to a linear function by Definition~\ref{def:equivalentfunction}, it follows that $\Delta$ is linear.

We now assume that $\Phi(H)$ is non-empty. We recall that $\Phi(H)$ is a subsurface of $M_H$, and it is possible that $\Phi(H)$ is disconnected. By Theorem~\ref{technical}, it suffices to compute the distortion of each component $S_i$ of $\Phi(H)$ in $M$.  Note that by our modified definition of almost fiber surface (Definition \ref{modified}), each $S_i$ does not contain any annulus piece. 

We recall that each $S_i$ is a union of virtual fiber surfaces and partial fiber surfaces along circles. Each partial fiber surface $B_{j}$ contained in $S_i$ is mapped into a piece $M_{H,j}$ of $M_H$. Note that some boundary components of $B_{j}$ are mapped to the boundary of  $M_{H,j}$, and some boundary components of $B_{j}$ are mapped to the interior of $M_{H,j}$. However, only those boundary components of $B_{j}$ that are mapped to the boundary of  $M_{H,j}$ are important for the generalized spirality character (see Section~3.3 in \cite{Sun18}) and subgroup distortion. We remark here that it is possible that the surface $S_i$ is not a clean surface as defined in \cite{Ngu18a}, since $B_j$ may not be a proper subsurface in the corresponding piece of $M$ (c.f. Definition 3.8 (1) of \cite{Ngu18a}). Although the main theorem in \cite{Ngu18a} (Theorem~1.2) is stated for clean surfaces, the proof of the main theorem in \cite{Ngu18a} still holds for the surface $S_i$. Thus, we have the following:
\begin{enumerate}[label=(\alph*)]
    \item If $S_i$ contains a geometrically infinite piece and $H_{i} = \pi_1(S_i)$ is non-separable in $G = \pi_1(M)$ then $\Delta_{H_i}^{G} \sim e^{e^n}$.
    \item If $S_i$ contains a geometrically infinite piece and $H_i$ is seprable in $G$ then $\Delta_{H_i}^{G} \sim e^n$.
    \item Assume none of the above, if $S_i$ contains two adjacent pieces, then $\Delta_{H_i}^{G}$ is exponential if $H_i$ is non-seprable in $G$ and $\Delta_{H_i}^{G}$ is quadratic if $H_i$ is seprable in $G$.
    \item Otherwise, $\Delta_{H_i}^{G}$ is linear.
\end{enumerate}
The proof is obtained by combining (a), (b), (c), (d) and Theorem \ref{technical}.
\end{proof}

For the rest of this section, we are going to prove Proposition~\ref{prop:distgeome} and Theorem~\ref{general}. 

We first need several lemmas to compute subgroup distortion of 3-manifolds with Sol and $\widetilde{SL(2,\mathbb{R})}$ geometries. The subgroup distortion in these two cases should be well-known for experts, but we can not find the literature on it.

The following elementary lemma is useful for the proofs of both of the geometries.
\begin{lem}
\label{lem:easy}
Let $G$ and $H$ be finitely generated groups with generating sets $\mathcal{A}$ and $\mathcal{B}$ respectively. Let $\phi \colon G \to H$ be a homomorphism. Then there exists a positive number $L$ such that $\bigabs{\phi(g)}_{\mathcal{B}} \le L\abs{g}_{\mathcal{A}}$ for all $g$ in $G$.
\end{lem}
\begin{proof}
Suppose that $\mathcal{A} = \{g_1,g_2,\dots ,g_n \}$ we define $L = \max \bigset{\abs{\phi(g_{i})}_{\mathcal{B}}}{i = 1,2,\dots,n}$. Since $\phi$ is a homomorphism, it is not hard to see that $\bigabs{\phi(g)}_{\mathcal{B}} \le L\abs{g}_{\mathcal{A}}$ for all $g \in G$.
\end{proof}

\begin{lem}[Subgroup distortion for Sol 3-manifolds]
\label{lem:distortioninSol}
Suppose that a 3-manifold $M$ has geometric structure modelled on Sol. We assume that $M$ is a torus bundle with Anosov monodromy (by passing to a double cover). Let $H$ be a finitely generated subgroup of $\pi_1(M)$. Then the distortion of $H$ in $\pi_1(M)$ is exponential if $H$ is an infinite subgroup in the fiber subgroup $\Z^2$ of $\pi_1(M)$, and it is linear otherwise.
\end{lem}
\begin{proof}
%By taking a double cover if necessary, we can assume that $M$ is a torus bundle with Anosov monodromy. Then
We note that $G = \pi_1(M)$ is the semi-direct product $\Z^2 \rtimes_\phi \Z$ where $\phi \in GL_{2}(\Z)$ is the matrix corresponding to the monodromy. We also note that $\Z$ is undistorted in $G$ and $\Z^{2}$ is exponentially distorted in $G$.

Let $M_H \to M$ be the covering space corresponding to $H$. The bundle structure on $M$ induces a bundle structure on $M_H$ that is either a bundle over the circle or a bundle over the line.

{\bf{Case 1:}} $M_H$ is a bundle over the circle. Let $\varphi \colon M_H \to S^1$ be the induced bundle map on $M_H$, with fiber surface $\Sigma_H$. The only possibilities for $\Sigma_H$ are plane, torus, and cylinder. We consider the following sub-cases:

Case~1.1: $\Sigma_H$ is the plane. It follows that $\varphi_{*} \colon H = \pi_1(M_H) \to \Z$ is bijective. So $H$ is undistorted in $G$, by Lemma \ref{lem:easy}.

Case~1.2: $\Sigma_H$ is the torus. Then $M_H$ is compact and the covering map $M_H \to M$ is a finite cover. It follows that $H$ is a finite index subgroup of $G =\pi_1(M)$. Thus $H$ is undistorted in $G$.

Case~1.3: $\Sigma_H$ is the cylinder. We note that this case actually can not happen since the monodromy does not have any rational eigenvector.

{\bf{Case 2:}} $M_H$ is a bundle over the line. Let $\varphi \colon M_H \to \R$ be the induced bundle map on $M_H$, with fiber surface $\Sigma_H$. Again, the only possibilities for $\Sigma_H$ are plane, torus, and cylinder. We consider the following sub-cases:

Case~2.1: $\Sigma_H$ is the plane. In this case, $H =\pi_1(M_H)$ is the trivial subgroup, and thus $H$ is undistorted in $G$.

Case~2.2: $\Sigma_H$ is the torus. In this case, we note that $\pi_1(M_H)$ is a finite index subgroup of the $\Z^2$ of $G= \Z^2 \rtimes_\phi \Z$. Thus, $H$ is exponentially distorted in $G$.

Case~2.3: $\Sigma_H$ is the cylinder. Then $H =\pi_1(M)$ is an infinite cyclic group contained in the $\Z^2$ subgroup of $G= \Z^2 \rtimes_\phi \Z$. We will show that $H$ is exponentially distorted in $G$. Indeed, the universal cover $\title{M} = \R^3$ of $M$ has the Sol metric defined by $ds^2 = \frac{dx^2}{e^{2z}} + e^{2z}dy^2 + dz^2$. We recall that any point on the $x$ axis of distance $n$ from the origin can be connected to the origin via a geodesic in the hyperbolic plane in the $x$ and $z$ direction (with positive $z$-coordinate), such that the length of this geodesic is approximated $\ln(n)$. Similarly for the $y$--axis, but the geodesic travels down instead up. Let $H$ be generated by $b \in \Z^2$. For any $n \in \N$, we can connect the origin and $nb$ by two geodesics as above (one travels up and the other travels down) with length approximate $|b|\ln(n)$. Thus, $H$ is exponentially distorted in $G$.
\end{proof}

The next proposition takes care of 3-manifolds with the $\widetilde{SL(2,\mathbb{R})}$-geometry.
\begin{prop}
\label{prop:SL2}
Let $M \to \Sigma$ be a circle bundle over a compact, connected, orientable surface $\Sigma$ with $\chi(\Sigma) <0$. Let $H$ be a finitely generated subgroup of $G =\pi_1(M)$, then $H$ is undistorted in $G$.
\end{prop}

%As $ M \to \Sigma$ be a circle bundle, we have an exact sequence 
%\[
%1 \to K \to \pi_1(M) \to \pi_1(\Sigma) \to 1
%\] where $K$ is the normal cyclic subgroup of $\pi_1(M)$ generated by a regular fiber. 

To prove Proposition \ref{prop:SL2}, we need a few lemmas.

The first lemma is well-known. The proof for the case of closed surfaces can be seen in Proposition 2 page 171 in \cite{NR93}. The proof for the case of surfaces with nonempty boundary is a corollary of Marshall Hall's Theorem.
\begin{lem}
\label{lem:quasiconvexsurface}
Let $S$ be a compact, connected, orientable surface with $\chi(S) <0$. Then every finitely generated subgroup of $\pi_1(S)$ is quasiconvex. In particular, this subgroup is undistorted in $\pi_1(S)$.
\end{lem}

\begin{lem}
\label{lem:rfact}
Let $S^1 \to M \to \Sigma$ be a circle bundle over a compact, connected, orientable surface $\Sigma$ with $\chi(\Sigma) <0$. Let \[
1 \to K \to \pi_1(M) \to \pi_1(\Sigma) \to 1
\] be the short exact sequence associated to the circle bundle where $K$ is the normal cyclic subgroup of $\pi_1(M)$ generated by a fiber. Then $K$ is undistorted in $\pi_1(M)$.
\end{lem}

\begin{proof}
Passing to finite cover if necessary, we only need to consider two cases:

Case~1: $\pi_1(M)$ is the product $\pi_1(\Sigma) \times K$. It is obvious that $K$ is undistorted in $\pi_1(M)$.

Case~2: The short exact sequence does not split, and thus $M$ has a geometry modelled on $\widetilde{SL(2,\mathbb{R})}$ that is an $\R$--bundle over $\mathbb{H}^2$. It is shown in \cite{Rie93} that there is a bi-lipschitz map from the unit tangent bundle $UT(\mathbb{H}^2) = \mathbb{H}^2 \times S^1$ of the hyperbolic plane to $\mathbb{H}^2 \times S^1$ that maps circles  to circles. By lifting to the universal cover, we have a quasi-isometry from $\widetilde{SL_2} \to \mathbb{H}^2 \times \R$ mapping lines to lines. Since subgroup distortion is invariant under quasi-isometry of pairs (see Definition~2.3 in \cite{Ngu18b}), and the $\R$ in $\mathbb{H}^2 \times \R$ is undistorted in $\mathbb{H}^2 \times \R$, it follows that the $\R$ in $\widetilde{SL_2}$ is undistorted. Thus, $K$ is undistorted in $\pi_1(M)$.
\end{proof}

\begin{lem}
\label{lem:easy2}
Let $G$ and $G'$ be two finitely generated groups. Let $H$ and $H'$ be two finitely generated subgroups of $G$ and $G'$ respectively. Suppose that $H'$ is undistorted in $G'$, and $\varphi \colon G \to G'$ is a homomorphism such that $\varphi_{|H} \colon H \to H'$ is an isomorphism. Then $H$ is undistorted in $G$.
\end{lem}
\begin{proof}
Fix generating sets of $G$ and $G'$.
Let $L$ be the constant given by Lemma~\ref{lem:easy} with respect to the homomorphism $\varphi \colon G \to G'$.
Let $\mathcal{A}$ be a finite generating set of $H'$. Let $\mathcal{S} = (\varphi|_H)^{-1}(\mathcal{A})$, it follows that $\mathcal{S}$ is a finite generating set of $H$.
For any $n \in \mathbb{N}$, let $h$ be an arbitrary element of $H$ such that $|h|_{G} \le n$. By Lemma~\ref{lem:easy}, we have 
$$\bigabs{\varphi(h)}_{G'} \le L \bigabs {h}_{G} \le Ln.$$
Since $H'$ is undistorted in $G'$, it follows that the inclusion $H' \to G'$ is $(C,0)$--quasi-isometric embedding for some constant $C$. It follows that $\bigabs{\varphi(h)}_{\mathcal{A}} \le C \bigabs{\varphi(h)}_{G'}  \le CLn$. Thus $\bigabs{h}_{\mathcal{S}} = \bigabs{\varphi(h)}_{\mathcal{A}} \le CLn$. The lemma is confirmed. 
\end{proof}

%%%%%%%%%%%%%%%%%%%%%%%%%%%%%%%%%%%%%%%%%%%%%%%%%%%%%%%%%%%%%
%%%%%%%%%%%%%%%%%%%%%%%%%%%%%%%%%%%%%%%%%%%%%%%%%%%%%%%%%%%%
%%%%%%%%%%%%%%%%%%%%%%%%%%%%%%%%%%%%%%%%%%%%%%%%%%%%%%%%%%%
%ALTERNATIVE PROOF IN CASE SL2

\begin{proof}[Proof of Proposition~\ref{prop:SL2}]
We have the short exact sequence
$$1 \to \Z \to \pi_1(M) \to \pi_1(\Sigma) \to 1$$
Let $ M_H \to M$ be the covering space corresponding to the subgroup $H \le \pi_1(M)$. The foliation into circles of $f \colon M \to \Sigma$ lifts to a foliation into circles or lines in $M_H$, with the base surface $\Sigma_H$. 

{\bf{Case~1:}} $M_H$ foliates in circles.
%Since $H$ is finitely generated subgroup of $\pi_1(M)$, it follows that $\pi_1(\Sigma_H)$ is finitely generated. Since every finitely generated subgroup of a surface group (either closed or with boundary) is quasiconvex, it follows that $f_{H*}(\pi_1(\Sigma_H))$ is quasiconvex in $\pi_1(\Sigma)$ (see Lemma~\ref{lem:quasiconvexsurface}).
We have the short exact sequence
\[
1 \to K  \to \pi_1(M_H) \to \pi_1(\Sigma_H) \to 1
\]
where $K$ is the cyclic subgroup of $H= \pi_1(M_H)$ generated by a fiber of $M_H$. %The surjective map $\pi_1(M) \to \pi_1(\Sigma)$ in the first sequence is the induced homomorphism $\sigma_{*}$ and the surjective map $\pi_1(M_H) \to \pi_1(\Sigma_H)$ in the second sequence is the induced homomorphism $\sigma_{H*}$. 
If $\pi_1(\Sigma_H)$ has finite index in $\pi_1(\Sigma)$ then $H$ is a finite index subgroup  of $\pi_1(M)$, thus $H$ is undistorted in $\pi_1(M)$. We now consider the case $\pi_1(\Sigma_H)$ has infinite index in $\pi_1(\Sigma)$. If $\pi_1(\Sigma_H)$ is trivial then $\pi_1(M_H) \cong K$. Since the $\Z$ is undistorted in $\pi_1(M)$ by Lemma~\ref{lem:rfact}, it follows that $K$ is undistorted in $\pi_1(M)$, so does $H$. If $\pi_1(\Sigma_H)$ is non-trivial then $\pi_1(\Sigma_H)$ is free, and thus the short exact sequence $1 \to K \to H \to \pi_1(\Sigma_H) \to 1$ splits.
By passing to a finite index subgroup, and by abusing notation, we can assume that $H \cong \pi_1(\Sigma_H) \times K$.

We first choose generating sets for $\pi_1(\Sigma)$, $\pi_1(\Sigma_H)$, and $\pi_1(M)$. We can assume that the generating set of $\pi_1(\Sigma_H)$ is the subset of the generating set of $\pi_1(M)$.
Applying Lemma~\ref{lem:easy} to the surjection $f_{*} \colon \pi_1(M) \to \pi_1(\Sigma)$, we have a constant $L$ satisfies the property in Lemma~\ref{lem:easy}. We note that the restriction of ${f_{*}}_{|\pi_1(\Sigma_H)} \colon \pi_1(\Sigma_H) \to f_{*}(\pi_1(\Sigma_H))$ is an isomorphism. Since $f_{*}(\pi_1(\Sigma_H))$ is quasiconvex in $\pi_1(\Sigma)$ (see Lemma~\ref{lem:quasiconvexsurface}), it follows that the inclusion $f_{*}(\pi_1(\Sigma_H)) \to \pi_1(\Sigma)$ is a $(\epsilon, 0)$--quasi-isometric embedding (see Corollary~3.6 in \cite{BH99}). To see $H$ is undistorted in $\pi_1(M)$, let $h = xk$ be an arbitrary element in $H$ such that $x \in \pi_1(\Sigma)$, $k \in K$ and $|h|_{\pi_1(M)} \le n$. We have $|x|_{\pi_1(\Sigma_H)} = |f_{*}(h)|_{\pi_1(\Sigma_H)} \le \epsilon |f_{*}(h)|_{\pi_1(\Sigma)} \le \epsilon\,L |h|_{\pi_1(M)} \le \epsilon Ln$. It follows that $|k|_{\pi_1(M)} = |x^{-1}h|_{\pi_1(M)} \le |x|_{\pi_1(M)} + |h|_{\pi_1(M)} \le |x|_{\pi_1(\Sigma_H)} + |h|_{\pi_1(M)} \le \epsilon L n + n$. We recall that $K$ is undistorted in $\pi_1(M)$. Hence $|k|_{K} \le A(\epsilon L n+n)$ for some constant $A$ (does not depend on $k$). Thus $|h|_{H}= |x|_{\pi_1(\Sigma_H)} + |k|_{K}$ is bounded above by a linear function. (Here the equality holds since $H$ is a direct product of $K$ and $\pi_1(\Sigma_H)$ and we choose the generating set of $H$ accordingly.) In other words, $H$ is undistorted in $\pi_1(M)$.

{\bf{Case~2:}} $M_H$ foliates in lines.
In this case the homomorphism $\pi_1(M_H
) \to \pi_1(\Sigma_H)$ is a isomorphism.  It follows that the restriction of $f_{*} \colon \pi_1(M) \to \pi_1(\Sigma)$ to $H$ is an isomorphism from $H$ to $\pi_1(\Sigma_H)$. We note that $\pi_1(\Sigma_H)$ is quasiconvex in $\pi_1(\Sigma)$, by Lemma \ref{lem:quasiconvexsurface}. Then the fact that $H$ is undistorted in $\pi_1(M)$ follows from Lemma~\ref{lem:easy2}.
\end{proof}

After seeing an early version of this paper, Martin Bridson provided us an alternative proof of Proposition~\ref{prop:SL2} that is shorter and more elegant than our proof. We describe here Bridson's argument for Proposition~\ref{prop:SL2}. 

\begin{proof}[An alternative proof of Proposition~\ref{prop:SL2}]
Let $p_* \colon \pi_1(M)\to \pi_1(\Sigma)$ be the homomorphism induced by the Seifert fibration $p\colon M\to \Sigma$. Since $M$ is a Seifert manifold, the subgroup $H$ is an induced central cyclic extension of a subroup $K$ of $\pi_1(\Sigma)$. 

Case~1: The central cyclic subgroup is trivial. Then $p_*|_H \colon H\to \pi_1(\Sigma)$ is injective. If $K=p_*(H)$ has finite index in $\pi_1(\Sigma)$, then $M$ has the geometry of $\mathbb{H}^2 \times \R$ (in the sense of Thurston) and $H$ has finite index in a direct factor of $\pi_1(M)$, thus $H$ is undistorted in $\pi_1(M)$. If $K=p_*(H)$ has infinite index in $\pi_1(\Sigma)$, then $K$ is free and $K \le \pi_1(\Sigma)$ is a virtual retraction (i.e, there exists a finite index subgroup $V$ of $\pi_1(\Sigma)$ containing $K$ and a homomorphism $r \colon V \to K$ such that $r|_K$ is the identity). We then extend this virtual retraction to get a virtual retraction of $\pi_1(M)$ onto $H$ (note that $H \cong K$). Thus, $H$ is undistorted in $\pi_1(M)$.

Case~2: The central cylic subgroup is nontrivial, thus it is an infinite cyclic group. If $K$ has finite index in $\pi_1(\Sigma)$, then $H$ has finite index in $\pi_1(M)$, hence $H$ is undistorted in $\pi_1(M)$. If $K$ has infinite index in $\pi_1(\Sigma)$, then by passing to a subgroup of finite index we can assume that $H \cong K \times \Z$. By \cite{Sco73}, $K$ is the fundamental group of a subsurface $S_H$ of a finite cover $\Sigma'$ of $\Sigma$, and $S_H$ is a retract of $\Sigma'$. Let $M'$ be the finite cover of $M$ obtained by taking the pull-back bundle via $\Sigma'\to \Sigma$, then we have $H<\pi_1(M')$. By \cite{Rie93}, there exists a quasi-isometry $\pi_1(M') \to \pi_1(\Sigma') \times \Z$ that preserves the fiber $\Z$ subgroup and its projection to $\pi_1(\Sigma')$ is the identity (it can be viewed geometrically as a bi-lipschitz map between $\tilde M$ and $\tilde{\Sigma} \times \R$ that preserves fibers and equals identity on the base surface), and note that it is not a group homomorphism. So we have the following sequence of maps $$H \hookrightarrow \pi_1 (M') \to \pi_1(\Sigma') \times \Z \to K \times \Z \cong H,$$ which equals identity on $H$ (the second arrow is only a quasi-isometry, instead of a group homomorphism). In other words, $H \cong K \times \Z$ is a quasi-retract of $\pi_1(M)$, thus $H$ is undistorted in $\pi_1(M)$.
\end{proof}

It is well-known that finitely generated subgroups in Nil 3-manifolds are either linearly distorted or quadratically distorted (see \cite{Osin01}). The purpose of the following lemma is to  give a geometric description for determining which subgroups give linear and quadratic distortion. 
\begin{lem}[Subgroup distortion for Nil 3-manifolds]
Suppose that a 3-manifold $M$ has geometric structure modelled on Nil. We assume that $M$ is a circle bundle over torus (by passing to a finite cover). Let $H$ be a finitely generated subgroup of $\pi_1(M)$. Then the distortion of $H$ in $\pi_1(M)$ is quadratic if and only if $H$ has infinite index in $\pi_1(M)$ and $H$ intersects with the fiber subgroup nontrivially. Otherwise, the distorion is linear.
\end{lem}

\begin{proof}
By taking a finite cover, we can assume that $M$ is an orientable circle bundle over the torus $T^2$. We have the short exact sequence
$$1 \to \Z \to \pi_1(M) \to \Z^2 \to 1$$
We note that the subgroup $\Z$ is quadratically distorted in $\pi_1(M)$. To see this, we write a presentation of $\pi_1(M)$ as $\langle k, a,b |[a,k] =1, [b,k]= 1, [a,b] =k^{e} \rangle$ where $e\in \Z\setminus \{0\}$ is the Euler number of the bundle and the $\Z$ subgroup is generated by $k$. It is easy to see that $\Z$ is quadratically distorted in $\pi_1(M)$ (for example, see Section~3.0.4 in \cite{Sisto}).
If $H$ has finite index in $\pi_1(M)$ then $H$ is undistorted in $\pi_1(M)$. In the rest of the proof, we will assume that $H$ has infinite index in $\pi_1(M)$.
Let $M_H \to M$ be the covering space corresponding to the subgroup $H < \pi_1(M)$. The foliation into circles of $f \colon M \to T^{2}$ lifts to a foliation into circles or lines in $M_H$, with base surface $\Sigma_H$. We note that $\pi_1(\Sigma_H)$ is undistorted in $\Z^2$ since $\Z^2$ is abelian (for example, see Proposition~8.98 in \cite{DK18}). We consider the following cases:

Case~1: $M_H$ foliates in lines. It follows that $H \cong \pi_1(\Sigma_H)$ and the restriction of the surjection $f_{*} \colon \pi_1(M) \to \pi_1(T^2)$ to $H$ is an isomorphism from $H$ to $\pi_1(\Sigma_H)$. Applying Lemma~\ref{lem:easy2} to $\varphi =f_{*}$, we have $H$ is undistorted in $\pi_1(M)$.

Case~2: $M_H$ foliates in circles. We have a short exact sequence
$$1 \to K \to H  \to \pi_1(\Sigma_H) \to 1$$ where $K$ is the cyclic subroup of $\pi_1(M_H)$ generated by a fiber of $M_H$. We note that $\Sigma_H$ is either a cylinder or plane ($\Sigma_H$ could not be a torus since $H$ has infinite index in $\pi_1(M)$). If $\Sigma_H$ is a plane, then $K \cong H$. Since the $\Z$ in $\pi_1(M)$ is quadratically distorted and $K$ is a finite index subgroup of $\Z$, it follows that $K$ is quadratically distorted in $\pi_1(M)$, and so does $H$. If $\Sigma_H$ is a cylinder, then by taking a finite index subgroup, we can assume that $H$ is the product $K \times \pi_1(\Sigma_H)$. We note that $K$ is undistored in $H$ and $K$ is quadratically distorted in $\pi_1(M)$. It follows that $$n^2 \preceq \Delta_{K}^{\pi_1(M)} \preceq \Delta_{H}^{\pi_1(M)} \circ \Delta^{H}_{K} \sim  \Delta_{H}^{\pi_1(M)}.$$ By \cite{Osin01}, $\Delta_{H}^{\pi_1(M)} \preceq n^2$ holds, so we have $n^2 \sim \Delta^{\pi_1(M)}_{H}$.

We note $H$ intersects the fiber subgroup $\Z$ of $\pi_1(M)$ trivially in Case~1 and nontrivially in Case~2. The lemma is confirmed.
\end{proof}

Now we are ready to prove  Proposition~\ref{prop:distgeome}.

\begin{proof}[Proof of Proposition~\ref{prop:distgeome}]
If the geometry of $M$ is either spherical, $S^{2} \times \R$, or Euclidean, then $G$ is virtually abelian. Since subgroup distortion is well behaved under taking finite index subgroup and subgroups in abelian groups are undistorted (for example, see Proposition~8.98 in \cite{DK18}), it follows that $\Delta_{H}^{G}$ is linear.

If the geometry of $M$ is hyperbolic, then it follows from the covering theorem (\cite{Can96}) and the subgroup tameness theorem (\cite{Agol04}, \cite{CG06}) that $\Delta_{H}^{G}$ is linear if $H$ is geometrically finite, and is exponential otherwise.

If the geometry of $M$ is Nil, it is well-known that the distortion $\Delta_{H}^{G}$ is either linear or quadratic (for example, see \cite{Osin01}).

If the geometry of $M$ is Sol then the distortion of $H$ in $\pi_1(M)$ is either linear or exponential by  Lemma~\ref{lem:distortioninSol}.

If the geometry of $M$ is $\mathbb{H}^{2} \times \R$, then $M$ is finitely covered by the trivial circle bundle $S \times S^1$ where $S$ is a hyperbolic surface. By Proposition~\ref{prop:SL2}, $\Delta_{H}^{G}$ is linear.

If the geometry of $M$ is $\widetilde{SL(2,\R)}$, then we claim that $\Delta_{H}^{G}$ is linear. Indeed, if $\partial M \neq \emptyset$, then $M$ has an $\mathbb{H}^2 \times \R$ structure, and thus $M$ is finitely covered by a trivial circle bundle over a surface with negative Euler characteristic. By Proposition~\ref{prop:SL2}, $\Delta_{H}^{G}$ is linear. In the case $M$ is closed, then $M$ is finitely covered by a nontrival circle bundle over a surface $S$ with $\chi(S) <0$. Again, by Proposition~\ref{prop:SL2}, we have that $\Delta_{H}^{G}$ is linear.
\end{proof}

We are now going to prove Theorem \ref{general}.

\begin{proof}[Proof of Theorem \ref{general}]
Since $\pi_1(M)$ is a finitely generated group, it follows from the Scott core theorem (\cite{Sco73}) that $M$ contains a compact codim-0 submanifold such that the inclusion map of the submanifold into $M$ is a homotopy equivalence. In particular, the inclusion induces an isomorphism on their fundamental groups. We therefore can assume that the manifold $M$ is compact. Since subgroup distortion does not change under taking finte index subgroups, we can assume that $M$ is orientable by passing to a double cover if necessary.

We can also assume that $M$ is irreducible and $\partial$-irreducible by the following reason:
Since $M$ is a compact, orientiable 3-manifold, it decomposes into irreducible, $\partial$-irreducible pieces $M_1, \cdots, M_k$  (by the sphere-disc decomposition).
In particular, $G =\pi_1(M)$ is a free product of fundamental groups of 3-manifolds $M_j$, i.e, $G =\pi_1(M_1) * \cdots *\pi_1(M_k)$. Let $G_i = \pi_1(M_i)$, and let $H$ be a finitely generated subgroup of $G$. It follows from Kurosh Theorem that $H \cong H_{1} * \cdots * H_{m} *F_{k}$ where $F_k$ is a free group and each $H_i$ is equals to $H \cap g_{i}G_{j_i}g_{i}^{-1}$ for some $g_i \in G$ and $j_i \in \{1,2, \cdots, k\}$. We remark here that $G$ is hyperbolic relative to the collection $\mathbb{P} = \{G_1, \cdots, G_k\}$, and $H$ is relative quasiconvex in $(G,\mathbb{P})$. According to Theorem~10.5 in \cite{Hru10}, the distortion function $\Delta_{H}^{G}$ equals the superadditive closure of the subgroup distortions of these finitely generated subgroups $H_i=H\cap g_iG_{j_i}g_i^{-1}< G_{j_i}$.  In other words, for the purpose of computing subgroup distortion, we only need to focus on the case where the manifold $M$ is compact, connected, orientable, irreducible, and $\partial$-irreducible.

%The distortion of a finitely generated subgroup in a compact, orientable, irreducible, $\partial$-irreducible 3-manifold group will be addressed in the rest of this proof. We will show that the distortion  is either linear, quadratic, exponential, double exponential. 

For the rest, we will assume that $M$ is compact, orientable, irreducible 3-manifold. We are going to show the distortion of a finitely generated subgroup $H$ of $\pi_1(M)$ is either linear, quadratic, exponential or double exponential. Once this claim is established, we can conclude that the only possibility for subgroup distortion of finitely generated 3-manifold groups is linear, quadratic, exponential and double exponential.

%Since subgroup distortion is well-defined under taking finte index subgroups. We refer the reader to Section~2.2 in \cite{Sun} for an explanation how to reduce the study of the $3$--manifold $M$ with finitely generated fundamental group to the case that the $3$--manifold $M$ is compact, orientable, irreducible.

{\bf{Case~1:}} $M$ has nontrivial torus decomposition.

Case~1.1: $M$ supports the Sol geometry. In this case, the distortion of $H$ in $\pi_1(M)$ is either linear or exponential by  Lemma~\ref{lem:distortioninSol}.

Case~1.2: $M$ does not support the Sol geometry. 

We first reduce to the case that $M$ has empty or tori boundary. Suppose that $M$ has a boundary component of genus at least $2$, by the proof in Section~6.3 of \cite{Sun18}, we can paste hyperbolic $3$-manifolds with totally geodesic boundaries to $M$ to get a $3$-manifold $N$ with empty or tori boundary. The new manifold $N$ satisfies the following properies:
\begin{enumerate}
    \item $M$ is a submanifold of $N$ with incompressible boundary.
    \item The torus decomposition of $M$ also gives the torus decomposition of $N$.
    \item Each piece of $M$ with a boundary component of genus at least $2$ is contained in a hyperbolic piece of $N$.
\end{enumerate}
So we have a sequence of subgroups $H<\pi_1(M)<\pi_1(N)$. Let $N_M$ be the covering space of $N$ corresponding to $\pi_1(M)<\pi_1(N)$, then $N_M$ has neither virtual fiber nor partial fiber pieces, since it only consists of finite cover pieces, geometrically finite pieces and simply connected pieces. Since $N$ has empty or tori boundary, Theorem \ref{toriboundary} implies that $\pi_1(M)$ is undistorted in $\pi_1(N)$. Then by a standard argument on distortions, we have that $\Delta^{\pi_1(M)}_H \sim \Delta^{\pi_1(N)}_H$ holds. Then by applying Theorem \ref{toriboundary} to $H<\pi_1(N)$, we have that $\Delta^{\pi_1(M)}_H \sim \Delta^{\pi_1(N)}_H$ can only be linear, quadratic, exponential or double exponential.

%According to the proof in Section~6.3 \cite{Sun18}, we also can assume that $M$ has empty or tori boundary.
The distortion of a finitely generated subgroup $H$ in the fundamental group of a compact orientable irreducible $3$--manifold with empty or tori boundary and nontrivial torus decomposition has been addressed in  Theorem \ref{toriboundary}, and the only possibility of the distortion is linear, quadratic, exponential, and double exponential. %If $M$ support Sol geometry, then the distortion of $H$ in $\pi_1(M)$ is either linear or exponential by  Lemma~\ref{lem:distortioninSol}.

{\bf{Case~2:}} $M$ has trivial torus decomposition. 

Case~2.1:  $M$ has empty or tori boundary.
In this case, $M$ has a geometric structure modelled on seven of the eight geometries: $S^{3}, \R^{3},  S^{2} \times \R, \mathbb{H}^{3},  Nil,   \mathbb{H}^{2} \times \R, \widetilde{SL(2,\R)}$. The distortion of $H$ in $\pi_1(M)$ is addressed in Proposition~\ref{prop:distgeome}.
% Let $H$ be a finitely generated subgroup of $G = \pi_1(M)$.

Case~2.2: $M$ has a higher genus boundary. Then $M$ supports a geometrically finite hyperbolic structure with infinite volume. This case follows from a similar filling argument as in case~1.2: we paste hyperbolic 3-manifolds with totally geodesic boundaries to $M$ to get a finite volume hyperbolic $3$-manifold $N$. If $H<\pi_1(N)$ has finite index, then the distortion $\Delta^{\pi_1(M)}_H$ is linear. Otherwise, in the sequence of subgroups $H<\pi_1(M)<\pi_1(N)$, each group is an infinite index subgroup of the following one. Then $H$ must be a geometrically finite subgroup of $\pi_1(N)$.
This is true because of that $H$ has infinite index in $\pi_1(M)$, while $\pi_1(M)$ has infinite index in $\pi_1(N)$, while such a pattern can not happen if $H$ is a geometrically infinite (virtual fiber) subgroup of $\pi_1(N)$.
So $\Delta^{\pi_1(N)}_H$ must be linear, and so does $\Delta^{\pi_1(M)}_H$ (by Lemma \ref{lem:preparationdist}).
\end{proof}

\end{document}